\documentclass[a4paper,11pt]{article}

\usepackage{bbm}
\usepackage[pdfborder={0 0 0}]{hyperref}
\usepackage[T1]{fontenc}
\usepackage{microtype}
\usepackage[left=3.4 cm,top=2.4cm,bottom=3.4cm,right=3.4cm]{geometry}
\usepackage{ytableau} 
\ytableausetup{boxsize=0.5em,centertableaux}
\usepackage{mathtools,array}
\usepackage{amsfonts,amssymb,amsthm}
\usepackage{pst-node}
\usepackage{tikz-cd} 
\usepackage{color}

\newtheorem{theorem}{Theorem}[section]
\newtheorem{itheorem}{Theorem}
\newtheorem{proposition}[theorem]{Proposition}
\newtheorem{lemma}[theorem]{Lemma}

\newtheorem{corollary}[theorem]{Corollary}
\newtheorem{definition}[theorem]{Definition}

\numberwithin{equation}{section}

\theoremstyle{remark}

\DeclareMathOperator{\End}{End}
\DeclareMathOperator{\GL}{GL}
\DeclareMathOperator{\id}{id}
\DeclareMathOperator{\supp}{supp}
\DeclareMathOperator{\Tr}{Tr}
\DeclareMathOperator{\ptr}{ptr}

\def\Thoma{\mathbb{T}}
\def\ThomaYB{\mathbb{T}_{\rm YB}}
\def\Young{\mathbb{Y}}
\def\Rquot{\R_0/\mathord{\sim}}

\def\TL{{\rm TL}}

\newcommand{\Cl}{\mathbb{C}}
\newcommand{\Rl}{\mathbb{R}}
\newcommand{\Nl}{\mathbb{N}}
\newcommand{\Zl}{\mathbb{Z}}

\newcommand{\Om}{\Omega}
\newcommand{\om}{\omega}
\newcommand{\la}{\lambda}
\newcommand{\eps}{\varepsilon}

\newcommand{\R}{\mathcal{R}}
\newcommand{\B}{\mathcal{B}}
\newcommand{\T}{\mathcal{T}}
\newcommand{\M}{\mathcal{M}}
\newcommand{\N}{\mathcal{N}}
\newcommand{\Hil}{\mathcal{H}}
\newcommand{\E}{\mathcal{E}}

\newcommand{\Rti}{\tilde{R}}
\newcommand{\dti}{\tilde{d}}
\newcommand{\Vti}{\tilde{V}}
\newcommand{\Rhat}{\hat{R}}

\newcommand{\ot}{\otimes}
\newcommand{\tp}[1]{^{\otimes #1}}    

\newcommand{\bigboxplus}{
  \mathop{
    \vphantom{\bigoplus} 
    \mathchoice
      {\vcenter{\hbox{\resizebox{\widthof{$\displaystyle\bigoplus$}}{!}{$\boxplus$}}}}
      {\vcenter{\hbox{\resizebox{\widthof{$\bigoplus$}}{!}{$\boxplus$}}}}
      {\vcenter{\hbox{\resizebox{\widthof{$\scriptstyle\oplus$}}{!}{$\boxplus$}}}}
      {\vcenter{\hbox{\resizebox{\widthof{$\scriptscriptstyle\oplus$}}{!}{$\boxplus$}}}}
  }\displaylimits 
}

\newcommand{\symm}{\Lambda}                                       
\newcommand{\qsymm}{\widehat{\symm}}                              
\newcommand{\poly}[1]{\mathsf{#1}}
\newcommand{\elsym}[1]{\poly{e}_{#1}}                            
\newcommand{\felsym}[2]{\poly{e}_{#1} \left(#2 \right)}
\newcommand{\comsym}[1]{\poly{h}_{#1}}                           
\newcommand{\fcomsym}[2]{\poly{h}_{#1} \left(#2 \right)}
\newcommand{\powsum}[1]{\poly{p}_{#1}}                           
\newcommand{\fpowsum}[2]{\poly{p}_{#1} \left(#2 \right)}
\newcommand{\schur}[1]{\poly{s}_{#1}}                           
\newcommand{\fschur}[2]{\poly{s}_{#1} \left(#2 \right)}

\newcommand{\partn}[1]{\left[#1\right]}

\newcommand{\deq}{\coloneqq}  %

\title{Yang-Baxter representations of the\\ infinite symmetric group}
\author{Gandalf Lechner, Ulrich Pennig, Simon Wood\\\\
\small School of Mathematics, Cardiff University}

\date{March 22, 2018}

\begin{document}

\maketitle

\begin{abstract}
	Every unitary involutive solution of the quantum Yang-Baxter equation (R-matrix) defines an extremal character and a representation of the infinite symmetric group $S_\infty$. We give a complete classification of all such Yang-Baxter characters and determine which extremal characters of $S_\infty$ are of Yang-Baxter form. 
	
	Calling two involutive R-matrices equivalent if they have the same character and the same dimension, we show that equivalence classes can be parameterized by pairs of Young diagrams, and construct an explicit normal form R-matrix for each class. Using operator-algebraic techniques (subfactors), we prove that two R-matrices are equivalent if and only if they have similar partial traces.
	
	Furthermore, we describe the algebraic structure of the set of equivalence classes of all involutive R-matrices, and discuss several families of examples. These include unitary Yang-Baxter representations of the Temperley-Lieb algebra at loop parameter $\delta=2$, which can be completely classified in terms of their trace and dimension.
	
	\medskip
	\noindent{\bf Mathematics Subject Classification:} 16T25, 20C32, 46L37
\end{abstract}

\section{Introduction}

The Yang-Baxter equation is an algebraic equation that plays a striking role in a remarkable number of seemingly disparate fields: Statistical mechanics \cite{Baxter:1972}, quantum mechanics \cite{Yang:1967}, braid groups, knot theory \cite{Jones:1987,Turaev:1988}, integrable quantum field theory \cite{AbdallaAbdallaRothe:2001}, quasitriangular Hopf algebras \cite{ChariPressley:1994}, and von Neumann algebras  \cite{Jones:1983_2}, to name but a few.

In its most basic form, the (quantum) Yang-Baxter equation is an equation for an endomorphism $R\in\End(V\ot V)$ on the tensor square of a vector space~$V$, namely
\begin{align}\label{eq:YBE}
 	(R\otimes \id_V)(\id_V\otimes R)(R\otimes \id_V)
 	=
 	(\id_V\otimes R)(R\otimes \id_V)(\id_V\otimes R)
\end{align}
as an equation in $\End(V\ot V\ot V)$. Several variations of this equation exist (see, for example, \cite{ChariPressley:1994,Jimbo:1989_2,LuYanZhu:2000}), but we shall only consider the form \eqref{eq:YBE}. 

As a nonlinear system of $(\dim V)^6$ equations for $(\dim V)^4$ unknowns, the Yang-Baxter equation is notoriously hard to solve in general. One rich source of solutions to \eqref{eq:YBE} is the theory of quantum groups, pioneered by Drinfeld \cite{Drinfeld:1986} and Jimbo \cite{Jimbo:1986}, which connects the Yang-Baxter equation to the representation theory of Lie algebras. 
A complete understanding or classification of all solutions of the Yang-Baxter equation has, however, not been reached.

A completely different approach to solving the Yang-Baxter equation has been put forward by Hietarinta. Emphasizing the point that an interesting structure of the set of solutions of the Yang-Baxter equation can only be expected after dividing by an appropriate equivalence relation, he succeeded in finding all solutions of \eqref{eq:YBE} for $\dim V=2$ up to a certain equivalence relation, with the help of computer algebra \cite{Hietarinta:1992_8}. But already for $\dim V=3$, this program was only partially successful \cite{Hietarinta:1993_3}.

\medskip

In this article, we develop a new approach to the Yang-Baxter equation in any dimension by considering a more natural equivalence relation, based on representation theory, on its set of solutions. Up to this equivalence, we give a complete classification of the special class of all unitary involutive R-matrices.

Here and hereafter, we consider \eqref{eq:YBE} over a finite dimensional vector space $V$, and agree to write $d=\dim V$ throughout\footnote{The Yang-Baxter equation  \eqref{eq:YBE} is of course also well defined for infinite dimensional Hilbert spaces~$V$, and in fact many interesting infinite dimensional solutions exist: In particular, a solution of the Yang-Baxter equation {\em with spectral parameter} can be rewritten as one without, but on an infinite dimensional base space (see, for example, \cite[Lemma 2.2]{HollandsLechner:2016}).}. We fix a scalar product on $V$ (and hence its tensor powers), denote the set of all unitary solutions of \eqref{eq:YBE} by $\R(V)$, and write $\R$ for the union of $\R(V)$ over all finite dimensional vector spaces $V$. As usual, the elements of~$\R$ are referred to as R-matrices. Given some $R\in\R(V)$, we refer to $d=\dim V$ as the dimension of $R$ -- although $R$ is an endomorphism of a space of dimension $d^2$ -- and to
$V$ as the base space of $R$.

As is well known, any $R\in\R(V)$ generates (unitary) representations $\rho_R^{(n)}$, $n\in\Nl$, of the braid groups $B_n$ on $V\tp{n}$, by representing the elementary braid\footnote{Recall that a presentation of the braid group $B_n$ on $n$ strands is given by $B_n=\langle b_1,\ldots,b_{n-1}\,:\,b_ib_{i+1}b_i=b_{i+1}b_ib_{i+1},\;b_ib_j=b_jb_i\,\text{ for } |i-j|>1\rangle.$}~$b_k$, $k\in\{1,\ldots,n-1\}$, as
\begin{align}\label{eq:ConstantYBE->Rep}
     \rho_R^{(n)}(b_k)
     \deq
     \id_V\tp{(k-1)}\ot R\ot \id_V\tp{(n-k-1)}
     \in\End V\tp{n}
     \,.
\end{align}
We set $\rho_R^{(1)}=\id_V$. The representations \eqref{eq:ConstantYBE->Rep} are used to define our equivalence relation on $\R$ below. This relation is also suggested by applications in integrable quantum field theory \cite{AlazzawiLechner:2016}. In a particular context, the same relation was considered in \cite{Gurevich:1986}, and a more restrictive version of it was essential in the computation of all solutions of the Yang-Baxter equation for $d=2$~\cite{Hietarinta:1992_8}.

\begin{definition}\label{Definition:R0-Equivalence}
     Two $R$-matrices $R,S\in\R$ are defined to be equivalent, denoted $R\sim S$, if and only if for each $n\in\Nl$, the representations $\rho_R^{(n)}$ and $\rho_S^{(n)}$ are equivalent.
\end{definition}

It is easy to check that two R-matrices are equivalent if and only if they have the same dimension and the same normalized character, introduced in Sect.~\ref{section:characters}.
Also note that $R\sim S$ implies similarity of these endomorphisms, $R\cong S$, because $\rho_R^{(2)}(B_2)$ is generated by $R$.

Simple examples of equivalent R-matrices can be produced as follows: Let $A\in\GL(V)$ be unitary and $F\in\End(V\ot V)$ be the tensor flip, defined by $F(v\ot w)=w\ot v$. Then, for any $R\in\R(V)$,
\begin{align}\label{eq:square-equivalence}
	R\sim (A\ot A)R(A^{-1}\ot A^{-1}),\qquad
	R\sim FRF.
\end{align}
In particular, $(A\ot A)R(A^{-1}\ot A^{-1})$ and $FRF$ are elements of $\R(V)$. But in general, the two transformations \eqref{eq:square-equivalence} do not generate the full equivalence class of $R$.

\medskip

In this paper, we focus on the important special case in which the representations $\rho_R^{(n)}$ factor through the surjective group homomorphism $B_n\to S_n$ onto the symmetric group $S_n$ of $n$ letters. This happens if and only if $R$ is {\em involutive}, $R^2=1$.

Such involutive R-matrices play a prominent role in various fields, ranging from symmetries of categories of vector spaces \cite{Lyubashenko:1987} over scattering operators in integrable quantum field theory (with a spectral parameter) \cite{AbdallaAbdallaRothe:2001} to representations of Thompson's group $\cal V$ \cite{Jones:2016_2} and recent constructions of non-commutative spaces \cite{Dubois-VioletteLandi:2017}. They also form the starting point for investigating $q$-deformed R-matrices with general quadratic minimal polynomial.

We write $\R_0(V)\subset\R(V)$, $\R_0\subset\R$ for the subset of involutive R-matrices.

\medskip

In the involutive case, $\rho_R^{(2)}\cong\rho_S^{(2)}$ is equivalent to $R$ and $S$ having the same dimension and trace because the only possible eigenvalues are $\pm1$. There is an old conjecture by Gurevich to the effect that equivalence classes of involutive R-matrices $R$ are even uniquely characterized by these two numbers, the dimension and the trace of $R$ \cite[p.~760]{Gurevich:1986}. Gurevich gave a proof of his conjecture in the case that one of the eigenvalues of $R$ has multiplicity 1.

However, our findings in this article imply that this conjecture is false in general: The full equivalence $R\sim S$ is a much stronger condition than having the same dimension and trace, which is reflected in the rich structure of $\Rquot$ that we find. We shall prove:

\begin{itheorem}\label{theorem:Rquot=YxY-Introduction}
	$\Rquot$ is in one to one correspondence with pairs of Young diagrams. Classes of R-matrices of dimension~$d$ correspond to pairs of Young diagrams with $d$ boxes in total.	
\end{itheorem}

Our analysis uses the fact that the $S_n$-representations $\rho_R^{(n)}$, $R\in\R_0(V)$, define a representation $\rho_R$ of the {\em infinite} symmetric group~$S_\infty$ (the group of all bijections of $\Nl$ that move only finitely many points) inside the infinite (algebraic) tensor product $\bigcup_n\End(V)\tp{n}$. As any infinite discrete group having no normal abelian subgroup of finite index, $S_\infty$ admits unitary representations which are not of type I \cite{Thoma:1964_2}, meaning that its irreducible representations are not classifiable in a reasonable manner \cite{KerovOlshanskiVershik:2004}. 
However, an explicit parameterization of its (normalized) extremal characters, corresponding to finite factor representations, is known from the work of Thoma \cite{Thoma:1964}. This parameterization depends on countably many continuous variables $0\leq\alpha_i,\beta_i\leq1$, $i\in\Nl$. These {\em Thoma parameters} are subject to two simple conditions defining a simplex $\Thoma$ which we recall in Sect.~\ref{section:characters}. For 
background  information on the representation theory of $S_\infty$, we refer readers to the recent monograph \cite{BorodinOlshanski:2017} and the literature cited therein. 

\medskip

The upshot of our approach is that any involutive R-matrix $R\in\R_0$ defines an extremal character $\chi_R$ of $S_\infty$, and $R,S\in\R_0$ are equivalent if and only if they have the same character and the same dimension. Such {\em Yang-Baxter characters} (Sect.~\ref{section:extremality}) can therefore be parameterized by a subset $\ThomaYB$ of Thoma's simplex~$\Thoma$, or, equivalently, in terms of the Hilbert-Poincar\'e series of $\rho_R$ \cite{Davydov:2000, Gurevich:1991}. By a faithfulness consideration and a theorem of Wassermann \cite{Wassermann:1981}, it is straightforward to show that $\ThomaYB\neq\Thoma$ (Sect.~\ref{section:faith}).

Despite the finite dimensional appearance of the Yang-Baxter equation, a full understanding of the subset $\ThomaYB\subset\Thoma$ and the equivalence relation $\sim$ seems to require tools from infinite dimensional analysis and operator algebras. In Sect.~\ref{section:subfactors}, we propose an approach based on subfactors arising from the subgroup $\{\sigma\in S_\infty\,:\,\sigma(1)=1\}\subset S_\infty$ in Yang-Baxter representations $\rho_R$. Independent of the Yang-Baxter equation, similar subfactors have been considered before, by Gohm and K\"ostler in a setting of noncommutative probability theory in \cite{GohmKostler:2010_2,GohmKostler:2011}, and by Yamashita in \cite{Yamashita:2012}.

We will show that any $R\in\R_0$, $R\neq\pm\id_{V\ot V}$, defines an inclusion of~II$_1$ factors. The trace-preserving conditional expectation of this subfactor can be computed, and we show that it is closely related to the partial trace of~$R$, $\ptr R=(\id_{\End V}\ot\Tr_V)(R)\in\End V$. The partial trace turns out to be a complete invariant of $\sim$:

\begin{itheorem}\label{theorem:PartialTraces-Introduction}
	Two R-matrices $R,S\in\R_0$ are equivalent if and only if they have similar partial traces, $\ptr R\cong \ptr S$.
\end{itheorem}

Our subfactor approach also completes the characterization of $\ThomaYB$ as a subset of $\Thoma$. Using the standard notation for Thoma parameters (recalled in Thm.~\ref{theorem:ThomaParameters}), we find the following result.

\begin{itheorem}\label{theorem:ThomaYB-Introduction}
	An extremal character of $S_\infty$ is a Yang-Baxter character if and only if its Thoma parameters $\{\alpha_i\}_i$, $\{\beta_j\}_j$ satisfy
	\begin{enumerate}
		\item Only finitely many $\alpha_i$, $\beta_j$ are non-zero.
		\item $\sum_i\alpha_i+\sum_j\beta_j=1$.
		\item All $\alpha_i$, $\beta_j$ are rational. 
	\end{enumerate}
	The Thoma parameters of an R-matrix of dimension $d$ satisfy $d\alpha_i,d\beta_i\in\Nl$.
\end{itheorem}

It is interesting to note that although the Yang-Baxter representations \eqref{eq:ConstantYBE->Rep} have a quite special form, their Thoma parameters form a dense subset of $\Thoma$. Several different approaches to characters and representations of $S_\infty$ exist in the literature, of which we mention here in particular the asymptotic character theory of Kerov and Vershik \cite{VershikKerov:1981,VershikKerov:1982}, and the works of Olshanski and Okounkov on spherical representations \cite{Olshanski:1990,Okounkov:1997}, partly in parallel with Wassermann \cite{Wassermann:1981}.

One half of the proof of Thm.~\ref{theorem:ThomaYB-Introduction} relies on the operator-algebraic techniques mentioned above, and the other half (showing that any set of Thoma parameters $\{\alpha_i\}_i$, $\{\beta_j\}_j$ satisfying the conditions $i)$-$iii)$ is realized by some Yang-Baxter character $\chi_R$) is based on a constructive procedure for generating R-matrices which we develop in Sect.~\ref{section:NormalForms}. The main idea of this construction is to find a good replacement for taking direct sums of R-matrices/representations which respects the Yang-Baxter equation \eqref{eq:YBE}. We define a suitable binary operation $\boxplus$ on $\R$ which enables us to construct an explicit normal form R-matrix for each equivalence class in~$\Rquot$. Together with a suitably defined tensor product of R-matrices and a compatible $\lambda$-operation, the quotient $\Rquot$ has the structure of a $\lambda$-semi ring.

\medskip

Because of the rationality and finiteness properties of $\ThomaYB$ spelled out in Thm.~\ref{theorem:ThomaYB-Introduction}, it is useful to switch to a parameterization in terms of the integer {\em rescaled Thoma parameters} $a_i:=d\alpha_i$, $b_i:=d\beta_i$. The two integer partitions defined by the $a_i$ and $b_i$ lie at the root of our classification result in  Thm.~\ref{theorem:Rquot=YxY-Introduction}. We also recast this structure in terms of spectral data, reminiscent of previous work on the representation theory of $S_\infty$ (see, in particular, Okounkov \cite{Okounkov:1997} and Kerov, Olshanski, and Vershik \cite{KerovOlshanskiVershik:2004}): The eigenvalues of the partial trace $\ptr R$ of $R\in\R_0$ uniquely determine the rescaled Thoma parameters (see Thm.~\ref{theorem:Young-Parameterization}~$ii)$ for details).

\medskip

In addition to these main results summarized in the  Theorems~\ref{theorem:Rquot=YxY-Introduction}--\ref{theorem:ThomaYB-Introduction}, we also investigate the $C^*$-algebraic and combinatorial properties of Yang-Baxter representations. In Sect.~\ref{section:YBReps}, we use $K$-theory to characterize the equivalence $R\sim S$ in terms of approximate unitary equivalence of the homomorphisms $\rho_R$, $\rho_S$ (Thm.~\ref{thm:equiv_relations}) and recover a result of Kerov and Vershik \cite{VershikKerov:1983} in our Yang-Baxter setting.  As $K_0(C^*S_\infty)$ is isomorphic to a quotient of the ring of symmetric functions \cite{VershikKerov:1983}, this also enables us to give an explicit formula for the decomposition of the Yang-Baxter representations $\rho_R^{(n)}$ into irreducibles in terms of symmetric functions (Prop.~\ref{Proposition:Multiplicities}), and to compute the Hilbert-Poincar\'e series of $\rho_R$.

Our final Sect.~\ref{section:examples} is devoted to a discussion of examples, including in particular Yang-Baxter representations of the Temperley-Lieb algebra \cite{TemperleyLieb:1971_2} at loop parameter $\delta=2$. Our results allow us to classify such representations completely in terms of the dimension and trace of the underlying R-matrix (Prop.~\ref{proposition:TL-Rs}), thereby complementing results of Gurevich and Bytsko \cite{Gurevich:1991,Bytsko:2015}.

\bigskip

While it is clear that Theorems~\ref{theorem:Rquot=YxY-Introduction}--\ref{theorem:ThomaYB-Introduction} do not hold verbatim for general R-matrices (dropping the involutivity assumption), we do expect that many aspects of our techniques and results can be generalized to other families of R-matrices, such as those underlying knot polynomials, for which $\rho_R$ factors through a Hecke algebra \cite{Jones:1987,Turaev:1988,Wenzl:1988_2}.

\section{R-matrices and characters of \texorpdfstring{\boldmath{$S_\infty$}}{S}}\label{section:characters}

The infinite symmetric group $S_\infty$ is the group of all bijections of $\Nl$ that move only finitely many points, a countable discrete group with infinite (non-trivial) conjugacy classes. A {\em character} of $S_\infty$ is defined as a positive definite class function $\chi:S_\infty\to\Cl$ that is {\em normalized} at the identity, $\chi(e)=1$. For example, the trivial representation has the constant character 1. The characters of $S_\infty$ form a simplex, the extreme points of which are called {\em extremal characters} (or indecomposable characters). An example of an extremal character is the Plancherel trace $\chi(\sigma)=\delta_{\sigma,e}$.

\medskip

Thoma found the following characterization of extremality of characters of~$S_\infty$, often called Thoma multiplicativity. In its formulation, we define the support of $\sigma\in S_\infty$ as the complement of the fixed points of $\sigma:\Nl\to\Nl$.

\begin{theorem}{\bf\cite{Thoma:1964}}\label{theorem:ThomaMultiplicativity}
     A character $\chi$ of $S_\infty$ is extremal if and only if for any $\sigma,\sigma'\in S_\infty$ with disjoint supports, it holds that $\chi(\sigma\sigma')=\chi(\sigma)\chi(\sigma')$.
\end{theorem}

Some elements of $S_\infty$ will appear repeatedly. We write $\sigma_{i,j}=(i,j)$ for two-cycles, and specifically $\sigma_k=\sigma_{k, k+1}$ for neighboring transpositions, the standard generators of $S_\infty$. General $n$-cycles will be denoted $c_n\in S_\infty$. In case a specific choice of $n$-cycle is necessary, we choose
\begin{subequations}
\begin{align}\label{eq:cycle-reps}
	c_n
	&=
	\sigma_{n-1}\cdots\sigma_2\sigma_1
	=
	\sigma_{1,2}\sigma_{1,3}\cdots\sigma_{1,n}\,,\\
	\sigma_{1,n}
	&=
	\sigma_{n-1}\cdots\sigma_2\sigma_1\sigma_2\cdots\sigma_{n-1}\,.
	\label{eq:2cyclesingenerators}
\end{align}
\end{subequations}
In view of the above theorem and the cycle decomposition of permutations, an extremal character $\chi$ of $S_\infty$ is uniquely determined by its values on $n$-cycles. For a general group element $\sigma\in S_\infty$, one then has $$\chi(\sigma)=\prod_{n\geq2}\chi(c_n)^{k_n},$$ where~$k_n$ is the number of $n$-cycles in the decomposition of $\sigma$ into disjoint cycles.

\subsection{Extremality and Thoma's parameterization}\label{section:extremality}

We now connect $S_\infty$ to R-matrices by showing that any R-matrix defines an extremal character and corresponding factor representation of $S_\infty$. We will be working with the infinite tensor product $\E_0\deq\bigotimes_{n\geq1}\End V$ (defined only algebraically at this point), with inclusions fixed by tensoring with $\id_V$ in the last factor. With the group inclusions $S_n\subset S_{n+1}\subset S_\infty$ defined by letting $\sigma\in S_n$ act on $\Nl$ by keeping all $j>n$ fixed, the system of representations $\rho_R^{(n)}$, $R\in\R_0(V)$ is coherent and defines a ${}^*$-homomorphism $\rho_R:\Cl[S_\infty]\to\E_0$. The generators $\sigma_i$, $i\in\Nl$, are mapped to
\begin{align}
	R_i\deq \rho_R(\sigma_i)=1\tp{(i-1)}\ot R\ot 1\ot \ldots\,,
\end{align}
where here and hereafter, we write $1$ instead of $\id_V$ when the base space is clear from the context. Note that $R_i$ can be viewed as an element of $\End V\tp{n}$ for $n\geq i+1$, or of $\E_0$.

We refrain from viewing $\rho_R$ as a representation on $\bigotimes_{n\geq1}V$, as the definition of this space depends on choices. Our $S_\infty$-representations will be defined by composing $\rho_R$ with the GNS representation of $\E_0$ with respect to its unique normalized trace\footnote{See Sec.~\ref{section:ProductStates} for different choices of states on $\E_0$.},
\begin{align}\label{eq:tau}
	\tau=\bigotimes_{n\geq1}\frac{\Tr_V}{d}:\E_0\to\Cl\,.
\end{align}

\begin{proposition}\label{proposition:chiR-extremal}
     Let $R\in\R_0(V)$. Then
     \begin{align}\label{eq:DefChiR}
	  \chi_R\deq \tau\circ\rho_R
     \end{align}
     is an extremal character of $S_\infty$. On an $n$-cycle $c_n$, $n\geq2$, it evaluates to
     \begin{align}\label{eq:YB-Character-on-n-cycles}
		\chi_R(c_n)=d^{-n}\,\Tr_{V\tp{n}}(R_1\cdots R_{n-1}),\qquad d=\dim V.
     \end{align}
\end{proposition}
\begin{proof}
     By standard properties of the trace, $\chi_R$ is a normalized positive class function. To show that $\chi_R$ is also extremal, we have to verify that it factorizes over permutations $\sigma,\sigma'\in S_\infty$ with disjoint supports (Thm.~\ref{theorem:ThomaMultiplicativity}). 
     
     Let $\sigma,\sigma'\in S_\infty$ have disjoint supports. Taking into account that $\chi_R$ is a class function, we may assume without loss of generality that $\supp\sigma\subset\{1,\ldots,n\}$ and $\supp\sigma'\subset\{n+1,\ldots,n+m\}$ for some $n,m\in\Nl$.
     
     Setting $N\deq n+m$, we then have $\rho_R^{(N)}(\sigma)=\rho_R^{(n)}(\sigma)\ot1\tp{m}$ and $\rho_R^{(N)}(\sigma')=1\tp{n}\ot\rho_R^{(m)}(\sigma')$. Using $\Tr_{V\ot W}(A\ot B)=\Tr_V(A)\Tr_W(B)$, we arrive at
     \begin{align*}
	  \chi_R(\sigma\sigma')
	  &=
	  d^{-N}\Tr_{V\tp{N}}((\rho_R^{(n)}(\sigma)\ot1\tp{m})(1\tp{n}\ot\rho_R^{(m)}(\sigma'))
	  \\
	  &=
	  d^{-n}\Tr_{V\tp{n}}(\rho_R^{(n)}(\sigma))\cdot d^{-m}\Tr_{V\tp{m}}((\rho_R^{(m)}(\sigma'))
	  \\
	  &=
	  \chi_R(\sigma)\chi_R(\sigma')\,,
     \end{align*}
     and the proof of extremality of $\chi_R$ is finished.
     
     For the second statement, we only need to note that $\rho_R$ \eqref{eq:ConstantYBE->Rep} maps the $n$-cycle $\sigma_1\sigma_2\cdots\sigma_{n-1}$ to 
     $R_1R_2\cdots R_{n-1}\in\End(V)\tp{n}$.     
\end{proof}

We will call the characters $\chi_R$, $R\in\R_0$, {\em Yang-Baxter characters} of $S_\infty$. As we just demonstrated, every Yang-Baxter character is extremal. We will see in the next section that the converse is not true: not every extremal character is Yang-Baxter. 

\bigskip

Using the representation theory of finite groups and the inductive limit definition of $S_\infty$, it follows from Prop.~\ref{proposition:chiR-extremal} that two R-matrices $R,S\in\R_0$ are equivalent in the sense of Def.~\ref{Definition:R0-Equivalence} if and only if they have the same character and the same dimension. (As we work with normalized characters, the dimension is not contained in the character.) Thus the dimension and the sequence of traces \eqref{eq:YB-Character-on-n-cycles} (indirectly) characterize the equivalence classes $\Rquot$.

\medskip

Thoma not only found a criterion for characterizing extremal characters, but also gave a classification in terms of an infinite dimensional simplex.

\begin{theorem}{\bf\cite{Thoma:1964}}\label{theorem:ThomaParameters}
  Let $\Thoma$ denote the collection of all sequences $\{\alpha_i\}_{i\in\Nl}$, $\{\beta_i\}_{i\in\Nl}$ of real numbers such that
  \begin{enumerate}
  \item $\alpha_i\geq0$ and $\beta_i\geq0$,
  \item $\alpha_i\geq\alpha_{i+1}$ and $\beta_i\geq\beta_{i+1}$,
  \item $\sum_i\alpha_i+\sum_j\beta_j\leq1$.
  \end{enumerate}
  The pairs of sequences $(\alpha,\beta)\in\Thoma$ are in bijection with extremal characters of $S_\infty$. On an $n$-cycle, the character $\chi$ corresponding to $(\alpha,\beta)\in\Thoma$ takes the value
  \begin{align}\label{eq:Thoma-Formula}
    \chi(c_n)=\sum_i\alpha_i^n+(-1)^{n+1}\sum_i\beta_i^n\,,\qquad n\geq2.
  \end{align}
\end{theorem}

We will call the parameters $(\alpha,\beta)\in\Thoma$ the {\em Thoma parameters} of a character. 

\medskip

As a consequence of these results, any $R\in\R_0$ defines a point $(\alpha,\beta)\in\Thoma$. In the following sections, we will be concerned with the problem of identifying the subset of all Thoma parameters of Yang-Baxter characters inside $\Thoma$.

Another important question is how to extract the Thoma parameters $(\alpha,\beta)$ from an involutive R-matrix. In view of \eqref{eq:YB-Character-on-n-cycles} and \eqref{eq:Thoma-Formula}, the parameters $(\alpha,\beta)\in\Thoma$ corresponding to $R\in\R_0$ are uniquely fixed by the system of equations
\begin{align}\label{eq:R-alphabeta}
	\sum_i\alpha_i^n+(-1)^{n+1}\sum_i\beta_i^n
	=
	d^{-n}\,\Tr_{V\tp{n}}(R_1\cdots R_{n-1})\,,\qquad n\geq2.
\end{align}
We will develop tools to compute $(\alpha,\beta)$ directly from $R$ in Sections~\ref{section:subfactors} and~\ref{section:NormalForms}.

\medskip

To conclude this section, let us mention a few simple examples of R-matrices. Clearly, $\pm1=\pm\id_{V\ot V}\in\R_0(V)$. It is well known and easy to check that also $\pm F\in\R_0(V)$ is an involutive solution of \eqref{eq:YBE}, where $F(v\ot w)=w\ot v$ is the tensor flip.  The Thoma parameters of these R-matrices are the following. 

\begin{center}
	\begin{tabular}{|r|l|}
		\hline
		$R$ & non-vanishing Thoma parameters\\
		\hline\hline
		$1$ & $\alpha_1=1$, independent of $d$\\
		\hline
		$-1$ & $\beta_1=1$, independent of $d$\\
		\hline
		$F$ & $\alpha_1=\ldots=\alpha_d=d^{-1}$\\
		\hline
		$-F$ & $\beta_1=\ldots=\beta_d=d^{-1}$\\
		\hline
	\end{tabular}
\end{center}

Since the R-matrices $R=1$ and $R=-1$ obviously give the trivial and alternating representation of $S_\infty$, respectively, the first two lines immediately follow from \eqref{eq:R-alphabeta}. The claimed parameters of $\pm F$ can be verified by computing $\Tr_{V\tp{n}}(F_1\cdots F_{n-1})=d$.

\subsection{Faithfulness}\label{section:faith}

Given an R-matrix $R\in\R_0$ of dimension $d$, the homomorphism $\rho_R$ restricts to a representation $\rho_R^{(n)}$ of $S_n$ on $V\tp{n}$, which has dimension $d^n$. This observation expresses that Yang-Baxter representations are small in comparison to the group algebra $\Cl[S_n]$, and leads to restrictions on the Thoma parameters of Yang-Baxter characters.

\begin{proposition}\label{proposition:faithfulness}
	Let $R\in\R_0$.
	\begin{enumerate}
		\item As a group homomorphism, $\rho_R$ is injective if and only if $R\neq\pm1$.
		\item As an algebra homomorphism, $\rho_R:\Cl[S_\infty]\to\E_0$ is not injective.
	\end{enumerate}
\end{proposition}
\begin{proof}
	$i)$ This is a general property of $S_\infty$. Clearly, if $R=\pm1$ then $\rho_R$ is not injective. Conversely, assume that $\rho_R$ is not injective and $\sigma\in S_\infty$ lies in the kernel, then $\sigma$ also lies in the kernel of $\rho_R|_{S_n}$ for $n$ sufficiently large. But for $n\geq5$, the only non-trivial proper normal subgroup of $S_n$ is the alternating group $A_n$. Thus $\ker \rho_R|_{S_n}$ contains at least $A_n$. This implies that the image of $\rho_R$ is either trivial or $\Zl_2$. In the case at hand, this means that $\rho_R$ is injective if and only if $R\neq\pm1$.

	$ii)$ $\rho_R$ restricts to an algebra homomorphism $\rho_R^{(n)}:\Cl[S_n]\to\End V\tp{n}$. As the dimensions of $\Cl[S_n]$ and \(\End(V^{\otimes n})\) are $n!$ and \(d^{2n}\), respectively, and $n!>d^{2n}$ for $n$ sufficiently large, it follows that $\rho_R^{(n)}$ cannot be injective. 
\end{proof}

The second part of this proposition implies that Yang-Baxter characters are never faithful. This observation allows us to make use of the following theorem due to Wassermann \cite[Thm.~III.6.5]{Wassermann:1981}.

\begin{theorem}{\bf \cite{Wassermann:1981}}\label{theorem:wassermann}
	Let $\chi$ be an extremal character of $S_\infty$ with Thoma parameters $(\alpha,\beta)\in\Thoma$. Then $\chi$ is faithful as a state of the group $C^*$-algebra $C^*S_\infty$ if and only if either $\sum_i\alpha_i+\sum_i\beta_i<1$, or $\sum_i\alpha_i+\sum_i\beta_i=1$ and infinitely many $\alpha_i$ or $\beta_i$ are non-zero.
\end{theorem}

In combination with Prop.~\ref{proposition:faithfulness}~$ii)$, this immediately implies the following result.

\begin{corollary}\label{corollary:YB->gamma+finite}
     Let $(\alpha,\beta)\in\Thoma$ be the Thoma parameters of a Yang-Baxter character $\chi_R$. Then $\sum_i\alpha_i+\sum_i\beta_i=1$, and only finitely many $\alpha_i$ or $\beta_i$ are non-zero.
\end{corollary}

We can now give a first example of an extremal non-Yang-Baxter character, namely the Plancherel trace $\chi(\sigma)=\delta_{\sigma,e}$. By \eqref{eq:Thoma-Formula}, the Plancherel trace has Thoma parameters $\alpha=\beta=0$ and therefore violates the condition $\sum_i\alpha_i+\sum_i\beta_i=1$. Its GNS representation is the left regular representation, which is too large to be of Yang-Baxter form.

\section{Yang-Baxter subfactors}\label{section:subfactors}

The Thoma parameters of a Yang-Baxter character have further properties, in addition to the ones spelled out in Cor.~\ref{corollary:YB->gamma+finite}. To extract these properties, and to derive a characterization of the equivalence relation $\sim$, we now switch to a setting involving von Neumann algebras. Specifically, we will consider subfactors \cite{Jones:1983_2,JonesSunder:1997} arising from the subgroup
\begin{align}\label {eq:subgroups}
     S_\infty^>\subset S_\infty\,,\qquad S_\infty^>\deq \{\sigma\in S_\infty\,:\,\sigma(1)=1\}.
\end{align}
Given an extremal character $\chi$ of $S_\infty$, we may view it as a tracial state on the group $C^*$-algebra $C^*S_\infty$ (we denote the state and the character by the same symbol). 

The GNS data of $(C^*S_\infty,\chi)$ will be denoted $(\Hil_\chi,\Om_\chi,\pi_\chi)$, and the von Neumann algebra generated by the representation $\M_\chi\deq \pi_\chi(C^*S_\infty)''$. Since $\chi$ is extremal, $\M_\chi$ is a (finite) factor --- it is trivial for the one-dimensional trivial and alternating representations, and hyperfinite of type II$_1$ in all other cases.

In our situation of Yang-Baxter representations, we have the homomorphism $\rho_R:\Cl[S_\infty]\to\E_0=\bigcup_n\End V\tp{n}$. Proceeding to the GNS representation $\pi_\tau$ of~$\E_0$ with respect to the trace $\tau$, we may weakly close $\E_0$ to $\E$ (a hyperfinite II$_1$ factor), and obtain the subfactor 
\begin{align}
	\M_R\deq \rho_R(\Cl[S_\infty])''\subset\E\,.
\end{align}
Since $\pi_\tau$ is faithful (in contrast to $\rho_R$ and $\pi_{\chi_R}$, see Prop.~\ref{proposition:faithfulness}), we suppress it in our notation and often write $\rho_R$ instead of $\pi_\tau\circ\rho_R$. We can canonically identify $\pi_\tau\circ\rho_R=\pi_{\chi_R}$, $\Om_\tau=\Om_{\chi_R}$, $\M_R=\M_{\chi_R}$, $\Hil_{\chi_R}=\overline{\M_R\Om_\tau}$. 

As an aside, let us mention that our equivalence relation $R\sim S$ implies the unitary equivalence of the representations 
\begin{align}\label{eq:GNS-equivalence}
	R\sim S\Longrightarrow \pi_\tau\circ\rho_R\cong\pi_\tau\circ\rho_S\,.
\end{align}
In fact, $R\sim S$ implies $\chi_R=\chi_S$ and hence $\pi_{\chi_R}=\pi_{\chi_S}$ --- since $\pi_{\chi_R}$ can be identified with the restriction of  $\pi_\tau\circ\rho_R$ to $\Hil_{\chi_R}$, \eqref{eq:GNS-equivalence} follows.

\bigskip

The subgroup \eqref{eq:subgroups} generates the von Neumann algebra 
\begin{align}\label{eq:NM-subfactor}
	\N_R\deq \rho_R(\Cl[S_\infty^>])''\subset\M_R\,.
\end{align}
As $S_\infty^>\cong S_\infty$, this is a (I$_1$ or II$_1$) subfactor. 

Gohm and K\"ostler \cite{GohmKostler:2010_2} and Yamashita \cite{Yamashita:2012} have independently analyzed the subfactor $\N_\chi\subset\M_\chi$ in the setting of general (not necessarily Yang-Baxter) extremal characters. They found that it is irreducible if and only if the parameters $(\alpha,\beta)$ have one of the following values:
\begin{enumerate}
	\item $\alpha_1=\ldots=\alpha_d=d^{-1}$ for some $d\in\Nl$,
	\item $\beta_1=\ldots=\beta_d=d^{-1}$ for some $d\in\Nl$,
	\item $\alpha_i=0$ and $\beta_i=0$ for all $i$.
\end{enumerate}
By comparison with our examples of R-matrices at the end of the preceding section, we see that the relative commutant $\N_R'\cap\M_R$ is trivial if and only if $R$ is equivalent to one of the four R-matrices $1,-1,F,-F$, of arbitrary dimension $d\in\Nl$. As we pointed out earlier, the last possibility iii) is realized by the Plancherel trace, which is not Yang-Baxter.

\medskip

To extract information about $R$ from the subfactor \eqref{eq:NM-subfactor}, we consider the unique $\tau$-preserving conditional expectation onto its relative commutant, 
\begin{align}
	E_R:\M_R\to\N_R'\cap\M_R\,.
\end{align}
The inclusion $\N_R\subset\M_R$ is replicated on the level of the infinite tensor product~$\E$: Here we consider the inclusion $\Cl\ot\End V\ot\End V\ot\cdots\subset\E$, the relative commutant of which is $\End V$, viewed as a subalgebra of $\E$ via the embedding $X\mapsto X\ot1\ot1\ot\cdots$. The corresponding $\tau$-preserving conditional expectation is the partial trace
\begin{align}\label{eq:E}
	E:\E\to\End V\,,\qquad E=\id_{\End V}\ot\,\tau\ot\tau\ot\cdots\;.
\end{align}
In the following arguments, we will also need the map
\begin{align}\label{eq:tau1}
	    \tau_1:\E\to\E,\qquad 
        \tau_1(a\ot b\ot c\ot\ldots)&:=\tau(a)\,b\ot c\ot\ldots,
\end{align}
and the canonical shift $s:\E\to\E$, $s(x)=\id_V\ot x$. Clearly $\tau_1\circ s=\id_\E$.

Another important element are the so-called partial shifts, defined as
\begin{align}\label{eq:gamma-m-def}
	\gamma_m(M)=\lim_{n\to\infty}R_{m+1}R_{m+2}\cdots R_n\cdot M\cdot R_n\cdots R_{m+2}R_{m+1}\,,\qquad m\in\Nl_0\,.
\end{align}
These limits exist in the strong operator topology for any $M\in\M_R$ \cite[Prop.~2.13]{GohmKostler:2010_2}, and define $\tau$-preserving endomorphisms of $\M_R$. We recall the well-known fact that on $\M_R$, the endomorphism $\gamma_0$ coincides with the shift $s$. Indeed, a straightforward calculation based on the Yang-Baxter equation shows 
\begin{align*}
 \gamma_0(\rho_R(\sigma_k))=\rho_R(\sigma_{k+1})=\id_V\ot\rho_R(\sigma_k)=s(\rho_R(\sigma_k)),\qquad k\in\Nl.
\end{align*}

Our following considerations will imply that the diagram
\begin{equation}
\begin{tikzcd}
	\End V \arrow[leftarrow, r, "E"] \arrow[hookleftarrow]{d} & \E \arrow[hookleftarrow]{d}
	\\
	\N_R'\cap\M_R \arrow[leftarrow, r, "E_R"] & \M_R
\end{tikzcd}
\label{diagram:commutingsquare}
\end{equation}
is a commuting square. To begin with, we show that we have the inclusion $\N_R'\cap\M_R\subset\End V$ on the left hand side.

\begin{proposition}\label{prop:relcomm} $\N_R'\cap\M_R\subset\End V\cong \End V\ot\Cl\ot\Cl\ldots\subset\E$.
\end{proposition}
\begin{proof}
    We define a linear map $\Gamma:\E\to\E$ by $\Gamma(X):=\tau_1(R_1XR_1)$. For any element $X\in\E_0$ in the algebraic infinite tensor product (only finitely many non-trivial tensor factors), there exists $N\in\Nl$ such that $\Gamma^n(X)\in\End V$ for all $n\geq N$. To treat general $X\in\E$, we note that $\Gamma$ satisfies 
    \begin{align*}
     \|\Gamma(X)\|\leq\|X\|,\qquad \|\Gamma(X)\|_2\leq\|X\|_2,\qquad X\in\E,
    \end{align*}
    where $\|X\|_2^2=\tau(X^*X)$ is the 2-norm defined by $\tau$. The first bound implies that the sequence $(\Gamma^n)_{n\in\Nl}$ has pointwise weak limit points $\hat\Gamma:\E\to\E$, and the second bound implies that also the limit point maps $\hat\Gamma$ satisfy $\|\hat\Gamma(X)\|_2\leq\|X\|_2$, $X\in\E$. Any such limit point $\hat\Gamma$ satisfies $\hat\Gamma(\E_0)\subset\End V$, and since $\E_0\subset\E$ is dense in 2-norm, we conclude $\hat\Gamma(\E)\subset\End V$.
    
    Now let $M\in\N_R'\cap\M_R$. Since $M\in\M_R$, we have $s(M)=\gamma_0(M)$, and since $M\in\N_R'$ commutes with $R_k$, $k\geq2$, we find
    \begin{align*}
     s(M)
     =
     \gamma_0(M)
     &= 
     \lim_{n\to\infty}R_1\cdots R_nMR_n\cdots R_1
     = R_1MR_1,
    \end{align*}
    and therefore $M=\tau_1(s(M))=\Gamma(M)$. We conclude $M=\hat\Gamma(M)\in\End V$.
\end{proof}

The key step of our argument is to show that $E_R$ and $E$ agree on $R_1=\rho_R(\sigma_1)\in\M_R$. As in \cite{Yamashita:2012}, we consider the subgroups $T_n=\{\sigma\in S_{n+1}\,:\,\sigma(1)=1\}\subset S_\infty$ and the von Neumann algebras generated by them, $\N_{R,n}\deq \rho_R(T_n)''\subset\M_R$. As $T_n\subset T_{n+1}$, this yields a descending chain of relative commutants, $n\in\Nl$,
\begin{align*}
	\M_R\supset (\N_{R,n}'\cap\M_R)\supset(\N_{R,n+1}'\cap\M_R)\supset(\N_R'\cap\M_R)\,,
\end{align*}
with corresponding conditional expectations $E_{R,n}:\M_R\to\N_{R,n}'\cap\M_R$. Since $T_n$ is finite, $E_{R,n}$ is simply given by averaging, 
\begin{align}
	E_{R,n}(M)=\frac{1}{n!}\sum_{\sigma\in T_n}\rho_R(\sigma)M\rho_R(\sigma^{-1}),
	\qquad M\in\M_R\,.
\end{align}
It is not hard to compute that for $M=R_1$, one gets \cite{Yamashita:2012}
\begin{align}
	E_{R,n}(R_1)=\frac{1}{n}\sum_{j=2}^{n+1}\rho_R(\sigma_{1,j})\,.
\end{align}

\begin{lemma}\label{lemma:ER=E}
 $E_R(R_1)=E(R_1)$.
\end{lemma}
\begin{proof}
 By definition of $E$ and $E_R$, we have $E(R_1)\in\End V$ and $E_R(R_1)\in\N_R'\cap\M_R$, respectively. But according to Prop.~\ref{prop:relcomm}, $\N_R'\cap\M_R\subset\End V$, so $E_R(R_1)\in\End V$ as well. It is therefore sufficient to show $\tau(XE_R(R_1))=\tau(XE(R_1))$ for all $X\in\End V$.
 
 By the definition of the right partial trace $E$,
 \begin{align*}
  \tau(XE(R_1))=\tau(XR_1).
 \end{align*}
 To calculate $\tau(XE_R(R_1))$, we use the fact that $E_{R,n}\to E_R$ as $n\to\infty$ in the $2$-norm given by $\tau$. This implies 
 \begin{align*}
    \tau(XE_R(R_1))&=
    \lim_{n\to\infty}\tau(XE_{R,n}(R_1))
    =
    \lim_{n\to\infty}\frac{1}{n}\sum_{j=2}^{n+1}\tau(X\rho_R(\sigma_{1,j})).
 \end{align*}
As $\sigma_{1,j}=\sigma_{j-1}\cdots \sigma_2\sigma_1\sigma_2\cdots \sigma_{j-1}$ \eqref{eq:2cyclesingenerators} and $X\in\End V$ commutes with $R_k$ for $k>1$, this simplifies to
\begin{align*}
 \tau(XE_R(R_1))
 &=
 \lim_{n\to\infty}\frac{1}{n}\sum_{j=2}^{n+1}\tau(XR_{j-1}\cdots R_2R_1R_2\cdots R_{j-1}) 
 =
 \tau(XR_1).
\end{align*}
 The proof is finished.
\end{proof}

With $E(R_1)=E_R(R_1)$, we now have concrete elements of the relative commutant $\N_R'\cap\M_R$ at our disposal. For $R=\pm1$ or $R=\pm F$, these partial traces are trivial, $E(\pm1)=\pm d\id_V$, $E(\pm F)=\pm\id_V$, as can be computed directly or inferred from the previously quoted result on irreducibility of $\N_R\subset\M_R$. 

However, for all R-matrices not equivalent to $\pm1,\pm F$, we get non-trivial partial traces $E(R_1)$. In fact, it was shown in \cite{GohmKostler:2011,Yamashita:2012} that $E_R(R_1)$ generates the relative commutant $\N_R'\cap\M_R$. This implies that for $R\not\sim\pm1,\pm F$, the expectation $E(R_1)$ is not a multiple of the identity.

\medskip

The partial trace $E(R_1)$ of the R-matrix turns out to be a complete invariant for the equivalence relation $\sim$. This is a consequence of the next theorem, which follows from the work of Gohm and K\"ostler, and our  Lemma~\ref{lemma:ER=E}. These authors prove it in a setting of noncommutative probability \cite{GohmKostler:2010_2}, building on their earlier work \cite{GohmKostler:2009_2,Kostler:2010_3} (see also \cite{GohmKostler:2011}). In our situation, only certain aspects of \cite{GohmKostler:2009_2,GohmKostler:2010_2,Kostler:2010_3,GohmKostler:2011} are needed, and we give a shortened proof for the sake of self-containedness. 

This proof makes use of the partial shifts $\gamma_m$ \eqref{eq:gamma-m-def}. By explicit calculation based on \eqref{eq:2cyclesingenerators} and the Yang-Baxter equation, one shows that \cite[Prop.~3.3]{GohmKostler:2010_2}
\begin{align}\label{eq:gamma-m}
     \gamma_m(\rho_R(\sigma_{1,n}))=\begin{cases}
                                     \rho_R(\sigma_{1,n}) & n<m+1\\
                                     \rho_R(\sigma_{1,n+1}) &n\geq m+1
                                    \end{cases}
     .
\end{align}
As $R_1=\rho_R(\sigma_{1,2})$, this implies in particular 
\begin{align}\label{eq:gamma-1}
     \gamma_1^p(R_1)&=\rho_R(\sigma_{1,p+2})\,,\qquad p\in\Nl\,.
\end{align}
It follows immediately from the definition \eqref{eq:gamma-m-def} of $\gamma_m$, $m\geq1$, that the relative commutant $\N_R'\cap\M_R$ is contained in the fixed point algebra $\M_R^{\gamma_m}$, and consequently the conditional expectation $E_R$ is invariant, $E_R\circ\gamma_m=E_R$. Gohm and K\"ostler proved that for $m=1$, equality holds: $\M_R^{\gamma_1}=\N_R'\cap\M_R$ \cite[Thm.~3.6~(iii)]{GohmKostler:2010_2}.

\begin{proposition}\label{proposition:GK-theorem}
	Let $c_n\in S_\infty$ be an $n$-cycle, $n\geq2$. Then
	\begin{align}\label{eq:ChiViaPartialTraces}
		\chi_R(c_n)=\tau(E(R_1)^{n-1}).
	\end{align}
\end{proposition}
\begin{proof}
	For $n=2$, the statement is a direct consequence of the definition of $E$. For the induction step, we consider the specific cycle $c_{n+1}=c_n\sigma_{1,n+1}$ \eqref{eq:cycle-reps}. Writing $C_n=\rho_R(c_n)$ as a shorthand, we note that $\gamma_n(C_n)=C_n$ for (see \eqref{eq:cycle-reps} and \eqref{eq:gamma-m}). As $E_R$ is invariant under $\gamma_n$, we obtain
	\begin{align*}
		E_R(C_{n+1})
		=
		E_R(\gamma_n(C_n\rho_R(\sigma_{1,n+1}))
		&=
		E_R(C_n\cdot\rho_R(\sigma_{1,n+2}))
		.
	\end{align*}
	In the same manner, we can now insert the endomorphism $\gamma_{n+1}$, which also leaves $C_n$ invariant, and maps $\rho_R(\sigma_{1,n+2})$ to $\rho_R(\sigma_{1,n+3})$. By iteration and \eqref{eq:gamma-1}, this gives $E_R(C_{n+1})=E_R(C_n\,\rho_R(\sigma_{1,n+p}))=E_R(C_n\,\gamma_1^{n+p-2}(R_1))$, $p\in\Nl$.
	
	Averaging over $p$ yields for any $N\in\Nl$
	\begin{align*}
		E_R(C_{n+1})
		&=
		E_R\left(C_n\cdot \frac{1}{N}\sum_{p=1}^N\gamma_1^p(\gamma_1^{n-2}(R_1))\right)
		.
	\end{align*}
	We may now use the ergodic theorem \cite{Petersen:1983}, stating here that for any $M\in\M_R$, the ergodic averages $N^{-1}\sum_{p=1}^N\gamma_1^p(M)$ converge strongly to the conditional expectation $E_R(M)$ onto the fixed point algebra $\M_R^{\gamma_1}=\N_R'\cap\M_R$ as $N\to\infty$ \cite[Thm.8.3]{Kostler:2010_3}. As $\gamma_1^{n-2}(R_1)\in\M_R$, and $E_R$ is continuous in the strong operator topology, we have  $E_R(C_{n+1})=E_R(C_n\cdot E_R(\gamma_1^{n-2}(R_1)))=E_R(C_n)\cdot E_R(\gamma_1^{n-2}(R_1))$. In view of the $\gamma_1$-invariance of $E_R$ and Lemma~\ref{lemma:ER=E}, the last term simplifies to $E_R(\gamma_1^{n-2}(R_1))=E_R(R_1)=E(R_1)$. 
	
	We thus have shown $E_R(C_{n+1})=E_R(C_n)\cdot E(R_1)$, which implies $E_R(C_n)=E(R_1)^{n-1}$ by induction. Evaluating in $\tau$ then gives the claimed result.
\end{proof}

These results can be used to show that \eqref{diagram:commutingsquare} is indeed a commuting square, but we will not need this fact in the following.

We have now extracted sufficient information from the subfactor setting, and return to our analysis of equivalence of R-matrices. At this point, it is better to switch to the usual partial trace of $R$, defined as
\begin{align}\label{eq:partialtracedef}
	\ptr R=\dim V\cdot E(R_1)=(\id_{\End V}\ot\Tr_V)(R)
\end{align}
and viewed as an element of $\End V$ rather than $\E$. The key relation \eqref{eq:ChiViaPartialTraces} can then be rewritten as
\begin{align}\label{eq:keyrelation-ptr}
	\chi_R(c_n)
	=
	d^{-n}\,\Tr_V(\ptr(R)^{n-1})\,.
\end{align}
With this information, we obtain the proof of Thm.~\ref{theorem:PartialTraces-Introduction} from the Introduction.

\begin{theorem}\label{theorem:PartialTraces}
	Two R-matrices $R,S\in\R_0$ are equivalent if and only if they have similar partial traces, $\ptr R\cong \ptr S$.
\end{theorem}
\begin{proof}
	If $R\in\R_0(V)$ and $S\in\R_0(W)$ have similar partial traces, then clearly  $\Tr_V(\ptr(R)^{n-1})=\Tr_W(\ptr(S)^{n-1})$. As similarity of the partial traces implies in particular that the dimensions $\dim V=\dim W$ coincide, we conclude $\chi_R=\chi_S$ from \eqref{eq:keyrelation-ptr} and Thoma multiplicativity. Thus, $\ptr(R)\cong\ptr(S)\Rightarrow R\sim S$.
	
	Conversely, if $R\sim S$, then these R-matrices have the same dimension and character, and hence $\Tr_V(\ptr(R)^{n-1})=\Tr_W(\ptr(S)^{n-1})$, $n\geq2$, from \eqref{eq:keyrelation-ptr}. This implies that the selfadjoint endomorphisms $\ptr(R)$, $\ptr(S)$ have the same characteristic polynomial, and are therefore similar.
\end{proof}

Elements $R\in\End(V\ot V)$ have two partial traces, the right one introduced above, and the left partial trace $\ptr'(R):=(\Tr_V\ot\id_{\End V})(R)$. We note that for R-matrices, these two partial traces coincide:
\begin{align}\label{eq:ptr=ptr'}
 \ptr(R)=\ptr'(R),\qquad R\in\R_0.
\end{align}
To prove this claim, we observe $\ptr'(R)=\ptr(FRF)$ with $F$ the flip, and recall $FRF\sim R$. Then Thm.~\ref{theorem:PartialTraces} implies the similarity $\ptr(R)\cong\ptr(FRF)=\ptr'(R)$ and in particular $\Tr(\ptr'(R)^2)=\Tr(\ptr(R)^2)$. Taking also into account $\Tr_V(\ptr(R)^2)=\Tr(R_1R_2)$ \eqref{eq:keyrelation-ptr} and $\Tr_V(\ptr(R)\ptr'(R))=\Tr(R_1R_2)$, we have 
\begin{align*}
 \Tr_V(|\ptr(R)-\ptr'(R)|^2)
 &=
 2\left(\Tr_V(\ptr(R)^2)-\Tr_V(\ptr(R)\ptr'(R))\right)
%
 =0,
\end{align*}
which proves \eqref{eq:ptr=ptr'}. We may therefore use left and right partial traces of R-matrices interchangeably in the following.

\medskip

Thm.~\ref{theorem:PartialTraces} shows that the eigenvalues of $\ptr(R)$ (and their multiplicities) characterize the equivalence classes $\Rquot$. Such spectral characterizations also appear in the work of Okounkov on Thoma measures and Olshanski pairs \cite{Okounkov:1997}. In our Yang-Baxter setting, the spectrum of $\ptr(R)$ has a very specific form, which will be the key to our classification of R-matrices in the next section.

\bigskip

As a second important consequence of Prop.~\ref{proposition:GK-theorem}, next we demonstrate that Yang-Baxter characters have {\em rational} Thoma parameters after stating a preparatory lemma.

\begin{lemma}\leavevmode
	\begin{enumerate}
		\item Let $\{x_i\}_i$ and $\{y_j\}_j$ be two finite sequences of positive real numbers such that for all $n\in\Nl$,
		\begin{align}\label{eq:rationality-source}
			\sum_ix_i^{2n+1}=\sum_j y_j^n\,.
		\end{align}
		Then the $x_i,y_j$ are rational. 
		\item Let $\{x_i\}_i$ be a finite sequence of positive rational numbers such that $\sum_ix_i^{2n+1}\in\Nl$ for all $n\in\Nl$. Then the $x_i$ are natural numbers.
	\end{enumerate}
\end{lemma}
\begin{proof}
	$i)$ We order the sequences $\{x_i\}_i$, $\{y_j\}_i$ non-increasingly and define $\mu\in\Nl$ as the multiplicity of the maximal value of the first sequence, i.e. $x_1=\ldots=x_\mu>x_{\mu+1}$. Dividing \eqref{eq:rationality-source} by $x_1^{2n+1}$ yields
	\begin{align*}
		\sum_i\left(\frac{x_i}{x_1}\right)^{2n+1}=\frac{1}{x_1}\sum_j \left(\frac{y_j}{x_1^2}\right)^n\,.
	\end{align*}
	In the limit $n\to\infty$, the left hand side converges to $\mu$. In this limit, the right hand side goes to infinity if $y_1>x_1^2$ and to $0$ if $y_1<x_1^2$. As $0<\mu<\infty$, we conclude that $y_1=x_1^2$, and define $\nu\in\Nl$ as its multiplicity, $y_1=\ldots=y_\nu>y_{\nu+1}$. Then the right hand side has the limit $\frac{\nu}{x_1}$ as $n\to\infty$, so that $x_1=\frac{\nu}{\mu}$ and $y_1=x_1^2$ are rational.
	
	Inserting these values of $x_1$ and $y_1$ into \eqref{eq:rationality-source}, we find
	\begin{align*}
		\mu \left(\frac{\nu}{\mu}\right)^{2n+1}+\sum_{i>\mu}x_i^{2n+1}
		=
		\nu \left(\frac{\nu}{\mu}\right)^{2n}+\sum_{j>\nu}y_j^n\,,
	\end{align*}
	and hence \eqref{eq:rationality-source} also holds for the shorter sequences $\{x_i\}_{i>\mu}$ and $\{y_j\}_{j>\nu}$. The claim now follows by induction.
	
	$ii)$ We first show that $\sum_i x_i^{2n}\in\Nl$, 
        that is, with {\em even} powers 
        implies\footnote{GL gratefully acknowledges a helpful
          discussion with A.~Schweizer on this point.} $x_i\in\Nl$. Assume
        that not all $x_i$ are integers. Then there exist $c,h_i\in\Nl$ and a
        prime $p$ such that $c\,x_i=h_i/p$ and not all $h_i$ are divisible
        by~$p$, that is, we clear all but at most a single
          factor of $p$ in the denominators of the $x_i$. 
        Clearly, it is sufficient to consider only those \(h_i\) which are not divisible by \(p\). Let \(\{g_i\}_i\) be the subsequence of \(\{h_i\}_i\) of elements not divisible by \(p\) and let its length be \(N\).        The assumption, $\sum_i x_i^{2n}\in\Nl$, then implies that $\sum_i g_i^{2n}=0\mod p^{2n}$ for every even $n\in\Nl$.
        
        For
        arbitrary $m\in\Nl$, consider Euler's totient function $\varphi(p^{m+1})=p^m(p-1)$ evaluated at $p^{m+1}$. Then, since $\varphi(p^{m+1})$ is even, 
        $S_m:=\sum_i g_i^{p^m(p-1)}= 0$ mod $p^{p^m(p-1)}$, and in particular, since $p^m\ge m$, $S_m= 0 \mod p^{m}$. However,
     since $p$ does not divide $g_i$,     the numbers $p^{m+1}$ and $g_i$ are coprime, and it follows from the Euler-Fermat theorem that $g_i^{p^m(p-1)}=1$ mod $p^{m+1}$. Thus $S_m=\sum_i  g_i^{p^m(p-1)}=N \mod p^{m+1}$, where we recall that $N$ denotes the number of terms in the sum. As 
$S_m=0$ mod $p^{m}$, it follows that $p^m$ divides $N$. Further, 
as $m$ was arbitrary, this implies $N=0$, that is, all $h_i$ are divisible by $p$ contradicting the initial assumption.

     We now deduce the claimed statement for {\em odd} powers. Let $x_i=\frac{u_i}{v_i}$ with $u_i,v_i\in\Nl$, and $\tilde{u}_i:=x_i\cdot v_1\cdots v_N\in\Nl$, and define a rational sequence $y_1=\ldots=y_{\tilde u_1}=x_1$, $y_{\tilde u_1+1}=\ldots=y_{\tilde u_1+\tilde u_2}=x_2$, etc. Then $\sum_j y_j^n=\sum_i \tilde u_ix_i^{n}=v_1\cdots v_n\sum_ix_i^{n+1}$, and by assumption, $\sum_ix_i^{n+1}\in\Nl$ for every {\em even} $n$, which in turn implies that
     $\sum_j y_j^n\in\Nl$ for all even $n$. 
     Thus all the \(y_j\) and \(x_i\) are natural numbers.
\end{proof}

\begin{definition}\label{definition:TYB}
	$\ThomaYB \subset\Thoma$ is defined as the subset of all $(\alpha,\beta)\in\Thoma$ satisfying:
	\begin{enumerate}
		\item Only finitely many parameters $\alpha_i,\beta_j$ are non-zero.
		\item $\sum_i\alpha_i+\sum_j\beta_j=1$.
		\item\label{definition:rational-parameters} All $\alpha_i,\beta_j$ are rational.
	\end{enumerate}
\end{definition}

We can now give one half of the proof of Thm.~\ref{theorem:ThomaYB-Introduction} from the Introduction.

\begin{theorem}\label{theorem:YB-parameters}
  The Thoma parameters $\alpha_i,\beta_i$ of any Yang-Baxter character $\chi_R$ lie in~$\ThomaYB$. If~$R$ has dimension $d$, then $d\alpha_i,d\beta_i\in\Nl$.
\end{theorem}
\begin{proof}
  Due to Cor.~\ref{corollary:YB->gamma+finite}, the only property of Definition \ref{definition:TYB} that remains to be shown is $\ref{definition:rational-parameters})$, the rationality of the parameters. To do so, we express the character on the left hand side of \eqref{eq:keyrelation-ptr} in terms of its Thoma parameters $(\alpha,\beta)$ \eqref{eq:Thoma-Formula}, and the traces on the right hand side of \eqref{eq:keyrelation-ptr} in terms of the non-zero eigenvalues $t_j$ of $\ptr R$ (note that $\ptr R$ is selfadjoint, so the $t_j$ are real). This yields, $n\ge2$,
	\begin{align}\label{eq:abt}
		\sum_i\alpha_i^n+(-1)^{n+1}\sum_i\beta_i^n
		=
		d^{-n}\sum_jt_j^{n-1}\,.
	\end{align}
	Specializing to the case that $n=2m+1$ is odd, we are in the situation of the preceding lemma with $\{x_i\}_i=\{d\alpha_i,d\beta_i\}_i$ and $y_j=t_j^2$, and conclude that the $\alpha_i,\beta_i$ are rational.
	
	To also show $d\alpha_i,d\beta_i\in\Nl$, note that $\sum_i((d\alpha_i)^{2n+1}+(d\beta_i)^{2n+1})	=\Tr_{V\tp{n}}(R_1\cdots R_{2n})$ are values of a non-normalized character of $S_{2n+1}$, and therefore integers. Thus $d\alpha_i,d\beta_i\in\Nl$ follows by application of the second statement of the preceding lemma.
\end{proof}

With a little more work, one can use \eqref{eq:abt} to show that the eigenvalues of $\ptr(R)$ are non-zero integers such that the positive eigenvalues coincide with the rescaled parameters $d\alpha_i$ and the negative eigenvalues coincide with the $-d\beta_i$, up to multiplicities. We will prove these facts by a different method in the next section.

\section{The structure of \texorpdfstring{\boldmath{$\Rquot$}}{R0}}

\subsection{Normal forms of involutive R-matrices}\label{section:NormalForms}

Our next aim is to prove that $\ThomaYB$ parameterizes the set of all Yang-Baxter characters, that is, that every $(\alpha,\beta)\in\ThomaYB$ is realized as the Thoma parameters of some R-matrix.

We will follow a procedure which has some analogy to building general group representations (of, say, a finite group) as direct sums of irreducibles. Yang-Baxter representations are reducible, but decomposing them gives representations which are no longer of Yang-Baxter form. Conversely, taking direct sums of Yang-Baxter representations is not compatible with the Yang-Baxter equation either.

To get around these problems, we introduce a binary operation $\boxplus$ on R-matrices that on the level of the base spaces corresponds to taking direct sums, and respects the Yang-Baxter equation. Under various names, such operations have been considered in the literature before  \cite{Lyubashenko:1987,Gurevich:1991,Hietarinta:1993_3}. We present here the version that is most useful for the case at hand.

\begin{definition}
     Let $V,W$ be finite dimensional Hilbert spaces and let $X\in\End(V\ot V)$, $Y\in\End(W\ot W)$. We define $X\boxplus Y\in\End((V\oplus W)\ot(V\oplus W))$ as
	\begin{align}\label{eq:DefBoxplus}
	       X\boxplus Y &= X\oplus Y\oplus F\quad\text{on}\\
	       (V\oplus W)\ot(V\oplus W) &= (V\ot V)\oplus(W\ot W)\oplus((V\ot W)\oplus(W\ot V)).\nonumber
	\end{align}	
\end{definition}

In other words, $X\boxplus Y$ acts as $X$ on $V\ot V$, as $Y$ on $W\ot W$, and as the flip on the mixed tensors involving factors from both, $V$ and $W$. Note that the above definition works in the same way for infinite dimensional Hilbert spaces.

\bigskip

Before applying this operation to R-matrices, we collect its main properties. In particular, we note that $\boxplus$ behaves well under taking the partial trace 
\begin{align}
   \ptr:\End(U\ot U)&\to\End U\,,\\
   X&\mapsto
   (\id_{\End U}\ot\Tr_U)(X),
\end{align}
where $U$ is any finite dimensional vector space.

\begin{lemma}\label{lemma:boxplus}
     Let $V,W$ be finite dimensional Hilbert spaces and $X\in\End(V\ot V)$, $Y\in\End(W\ot W)$.     
     \begin{enumerate}
	  \item $\boxplus$ is commutative and associative up to canonical isomorphism.
	  \item If $X$ and $Y$ are unitary (respectively selfadjoint, involutive, invertible), then $X\boxplus Y$ is unitary (respectively selfadjoint, involutive, invertible).
	  \item If $X$ commutes with the flip (on $V\ot V$) and $Y$ commutes with the flip (on $W\ot W$), then $X\boxplus Y$ commutes with the flip (on $(V\oplus W)\ot(V\oplus W)$).
	  \item $\ptr(X\boxplus Y)=(\ptr X)\oplus(\ptr Y)$. In particular, $\Tr(X\boxplus Y)=\Tr X+\Tr Y$. The same formula holds for the right partial trace.
     \end{enumerate}
\end{lemma}
\begin{proof}
     $i)$ The definition \eqref{eq:DefBoxplus} is invariant under exchanging $(X,V)$ with $(Y,W)$, that is $X\boxplus Y=Y\boxplus X$. Associativity follows by repeatedly evaluating the definition. Given finite dimensional vectors spaces $V^1,\ldots,V^n$ and $X^i\in\End(V^i\ot V^i)$, $i=1,\ldots,n$, one finds
     \begin{align}\label{eq:Boxplus-Iterated}
	  \bigboxplus_{i=1}^nX^i
	  &=
	  X^1\oplus\ldots\oplus X^n\oplus F\,,
     \end{align}
     where on the right hand side, each $X^i$ acts on $V^i\ot V^i$, and $F$ on the orthogonal complement of $\bigoplus_i (V^i\ot V^i)$ in $(\bigoplus_i V^i)\tp{2}$.

     $ii)$, $iii)$ These statements follow directly from the facts that $F$ is unitary, selfadjoint, involutive, invertible, and the flip of $(V\oplus W)\ot(V\oplus W)$ leaves the three subspaces in the decomposition \eqref{eq:DefBoxplus} invariant.
     
     $iv)$ Proving the claimed formula amounts to showing that the partial trace of $FQ$ vanishes, where $Q$ is the orthogonal projection onto $(V\ot W)\oplus(W\ot V)$. Let $v_1,v_2\in V$, and let $\{w_k\}$ be an orthonormal basis of $W$. Then the right partial trace satisfies $\langle v_1,\ptr(FQ)v_2\rangle=\sum_k\langle v_1\ot w_k, F(v_2\ot w_k)\rangle$, because $Q$ vanishes on $V\ot V$. But $\langle v_1\ot w_k, F(v_2\ot w_k)\rangle=0$ because $V$ and $W$ lie orthogonal to each other. The argument for the right partial trace is the same.
\end{proof}

We now apply $\boxplus$ to R-matrices. 
The following result is known \cite{Lyubashenko:1987,Gurevich:1991,Hietarinta:1993_3}. But since no proof seems to be available in the literature, we state it here with a proof.

\begin{proposition}\label{proposition:Boxplus-preserves-YBE}
	Let $R \in \R(V)$, $\Rti\in\R(\Vti)$. Then $R\boxplus\Rti\in\R(V \oplus \Vti)$.
\end{proposition}
\begin{proof}
     Invertibility of $R\boxplus\Rti$ follows from the preceding lemma and the invertibility of $R,\Rti$. The main point is to check that $\Rhat\deq R\boxplus\Rti$ solves the Yang-Baxter equation on $(V\oplus\Vti)\tp{3}$.
     
     This space is the direct sum of eight orthogonal subspaces $V_1\ot V_2\ot V_3$, where each $V_i$ is either $V$ or $\Vti$. Observe that both of the operators $\Rhat_1\Rhat_2\Rhat_1$ and $\Rhat_2\Rhat_1\Rhat_2$ decompose into direct sums of their restrictions to $V\tp{3}$,  $V\ot\Vti\ot V$, $(V\ot V\ot\Vti)\oplus(\Vti\ot V\ot V)$, $\Vti\tp{3}$, $\Vti\ot V\ot \Vti$ and $(\Vti\ot\Vti\ot V)\oplus(V\ot\Vti\ot\Vti)$. 
     
     By symmetry in $V, \Vti$ and $R,\Rti$ it suffices to show that $\Rhat_1\Rhat_2\Rhat_1$ and $\Rhat_2\Rhat_1\Rhat_2$ coincide on the first three subspaces in the list. Since $R \in \R(V)$, this is true for $V\tp{3}$. Inserting the definition of $\Rhat$, one finds that on $V\ot\Vti\ot V$, $\Rhat_1\Rhat_2\Rhat_1$ acts as $F_1R_2F_1$, while $\Rhat_2\Rhat_1\Rhat_2$ acts as $F_2R_1F_2$. These two operators coincide with the one acting as $R$ on the first and third tensor factors.
     
     Let $W\deq (V\ot V\ot\Vti)\oplus(\Vti\ot V\ot V)$. The restrictions of both sides of the Yang-Baxter equation evaluate to $\Rhat_1\Rhat_2\Rhat_1|_{W}=(F_1F_2R_1\oplus R_1F_2F_1)|_W$ and $\Rhat_2\Rhat_1\Rhat_2|_W=(R_2F_1F_2\oplus F_2F_1R_2)|_W$. The operator $F_1F_2$ coincides with the tensor flip on $(V \otimes V) \otimes \Vti$. This implies $F_1F_2R_1=R_2F_1F_2$. Likewise $F_2F_1$ is the tensor flip on $\Vti \otimes (V \otimes V)$ and we have $F_2F_1R_2=R_1F_2F_1$. Therefore $\Rhat_1\Rhat_2\Rhat_1|_{W} = \Rhat_2\Rhat_1\Rhat_2|_W$, which finishes the proof. 
\end{proof}

By Lemma~\ref{lemma:boxplus}~$ii)$, $\boxplus$ preserves involutivity, and thus also induces a binary operation on $\R_0\subset\R$. 

It is clear that variants of this operation are possible: A trivial change would be to use $-F$ instead of $F$ in the definition of $\boxplus$, but also more substantial variations exist \cite{Hietarinta:1993_3}. However, all these variations lead to R-matrices that are equivalent in the sense of Def.~\ref{Definition:R0-Equivalence}.

For characterizing equivalence classes of R-matrices, we next describe how $\boxplus$ acts on the Yang-Baxter characters of $S_\infty$ and their Thoma parameters.

\begin{proposition}\label{proposition:boxplus+thoma}
     Let $R,\Rti\in\R_0$ have dimensions $d,\tilde d$.  
 	\begin{enumerate}
		\item The characters of $R$, $\tilde R$, and $R\boxplus\Rti$ are related by ($c_n$ an $n$-cycle, $n\geq2$)
		\begin{align}\label{eq:boxplusontraces}
			\chi_{R\boxplus\Rti}(c_n)
			=
			\frac{d^n}{(d+\tilde d)^n}\,\chi_R(c_n)
			+
			\frac{\tilde d^n}{(d+\tilde d)^n}\,\chi_{\tilde R}(c_n)\,.
		\end{align}
		\item Let $(\alpha,\beta)$ and $(\tilde\alpha,\tilde\beta)$ be the Thoma parameters of $R$ and $\tilde R$, respectively. Then the Thoma parameters of $R\boxplus\Rti$ are the non-increasing arrangements of
		\begin{align}
			\{\hat\alpha_i\}_i
			&=
			\{\tfrac{d}{d+\tilde d}\,\alpha_k,\,\tfrac{\tilde d}{d+\tilde d}\,\tilde\alpha_l,\,:\,k,l\in\Nl\}
			\,,\nonumber\\
			\{\hat\beta_i\}_i
			&=
			\{\tfrac{d}{d+\tilde d}\,\beta_k,\,\tfrac{\tilde d}{d+\tilde d}\,\tilde\beta_l,\,:\,k,l\in\Nl\}\,.
			\label{eq:ThomaParametersOfBoxplus}
	\end{align}
	\end{enumerate}
\end{proposition}
\begin{proof}
  $i)$ We denote the base spaces of $R$ and $\tilde R$ by $V_+$ and $V_-$, respectively, and write $\Rhat\deq R\boxplus\Rti$ and $\hat V\deq V_+\oplus V_-$. Noting that the dimension of $\hat V$ is $d+\dti$, equation \eqref{eq:boxplusontraces} is equivalent to
     \begin{align}\label{eq:Rhat-traces}
	  \Tr_{\hat{V}\tp{n}}(\Rhat_1\cdots\Rhat_{n-1})
	  &=
	  \Tr_{V_+\tp{n}}(R_1\cdots R_{n-1})
	  +
	  \Tr_{V_-\tp{n}}(\Rti_1\cdots\Rti_{n-1})\,.
     \end{align}
     The trace on the left hand side is taken over $\hat V\tp{n}=\bigoplus_{\eps_1,\ldots,\eps_n}(V_{\eps_1}\ot\ldots\ot V_{\eps_n})$, where the sum runs over $\eps_i=\pm$, $i=1,\ldots,n$. We claim that 
     \begin{align}\label{eq:Rhat-Veps}
	  (\Rhat_1\cdots\Rhat_{n-1})V_{\eps_1}\ot\ldots\ot V_{\eps_n} \not\perp V_{\eps_1}\ot\ldots\ot V_{\eps_n}
	  \Rightarrow
	  \eps_1=\ldots=\eps_n\,.
     \end{align}
     Note that \eqref{eq:Rhat-Veps} implies \eqref{eq:Rhat-traces}: If \eqref{eq:Rhat-Veps} holds, then the trace over $\hat V\tp{n}$ simplifies to the sum of the trace over $V_+\tp{n}$ and that over $V_-\tp{n}$. As $\Rhat$ acts as $R$ and $\Rti$ on $V_+\ot V_+$ and $V_-\ot V_-$, respectively, \eqref{eq:Rhat-traces} then follows.
     
     To show \eqref{eq:Rhat-Veps}, we consider the position of the image $(\Rhat_1\cdots\Rhat_{n-1})V_{\eps_1}\ot\ldots\ot V_{\eps_n}$ relative to $V_{\eps_1}\ot\ldots\ot V_{\eps_n}$ for given $\eps_1,\ldots,\eps_n=\pm$. Assume that $\eps_{n-1}\neq\eps_n$. Then the rightmost factor $\Rhat_{n-1}$, acting non-trivially only on $V_{\eps_{n-1}}\ot V_{\eps_n}$, simplifies to the flip by definition of $\Rhat=R\boxplus\Rti$. As all other factors $\Rhat_1\cdots\Rhat_{n-2}$ act trivially on the last tensor factor $V_{\eps_n}$, this implies $(\Rhat_1\cdots\Rhat_{n-1})V_{\eps_1}\ot\ldots\ot V_{\eps_n}\subset\hat V\tp{n-1}\ot V_{\eps_{n-1}}\perp V_{\eps_1}\ot\ldots\ot V_{\eps_n}$. Hence the non-orthogonality assumption in \eqref{eq:Rhat-Veps} implies $\eps_{n-1}=\eps_n$.
     
     We next assume $\eps_{n-2}\neq\eps_{n-1}=\eps_n$. In this situation, the rightmost factor $\Rhat_{n-1}$ maps the product of the last two tensor factors $V_{\eps_n}\ot V_{\eps_n}$ onto itself, so that we are left with the same situation as before, but with the number of tensor factors reduced by one. Inductively, we conclude that the non-orthogonality assumption in \eqref{eq:Rhat-Veps} implies $\eps_1=\ldots=\eps_n$.     
     
     $ii)$ Define parameters $\hat\alpha_i$, $\hat\beta_j$ by \eqref{eq:ThomaParametersOfBoxplus}, ordered non-increasingly. Then $0\leq\hat\alpha_i,\hat\beta_j\leq1$, and for any $n\in\Nl$,
     \begin{align*}
	  \sum_i\hat\alpha_i^n+(-1)^{n+1}\sum_j\hat\beta_j^n
	  &=
	   \left(\frac{d}{d+\dti}\right)^n
	   \left(\sum_i\alpha_i^n+(-1)^{n+1}\sum_j\beta_j^n\right)\\
	   &\qquad+\left(\frac{\dti}{d+\dti}\right)^n
	   \left(\sum_i\tilde\alpha_i^n
	  +(-1)^{n+1}\sum_j\tilde\beta_j^n\right)\,.
     \end{align*}
     Since $(\alpha,\beta), (\tilde\alpha,\tilde\beta)\in\Thoma$, we have $\sum_i\alpha_i+\sum_j\beta_j\leq1$ and $\sum_i\tilde\alpha_i+\sum_j\tilde\beta_j\leq1$, and therefore $\sum_i\hat\alpha_i+\sum_j\hat\beta_j\leq1$. This shows that $(\hat\alpha,\hat\beta)\in\Thoma$. In terms of characters, the above equation reads, $n\in\Nl$,
     \begin{align*}
	   \sum_i\hat\alpha_i^n+(-1)^{n+1}\sum_j\hat\beta_j^n=\frac{d^n}{(d+\tilde d)^n}\,\chi_R(c_n) 
	   +\frac{\tilde d^n}{(d+\tilde d)^n}\,\chi_{\tilde R}(c_n)
	   \,,
     \end{align*}
     and by part $i)$ and the uniqueness of the Thoma parameters of an extremal character, identifies $(\hat\alpha,\hat\beta)$ as the Thoma parameters of $\chi_{\Rhat}$.
\end{proof}

By construction, $\boxplus$ maps pairs of parameters in $\ThomaYB $ into $\ThomaYB$, preserving the three properties of $\ThomaYB$ (Def.~\ref{definition:TYB}). But given $d,\tilde d>0$, \eqref{eq:boxplusontraces} also makes sense as an operation on general extremal characters of $S_\infty$. We do not investigate this observation any further here.

\bigskip

After these preparations, we come to the definition of special normal form R-matrices as $\boxplus$-sums of identities and negative identities. We will write $1_a$ for the identity on a vector space of dimension $a^2$, i.e. $1_a\in\R_0(\Cl^a)$.

\begin{definition}
     Let $n,m\in\Nl_0$ with $n+m\geq1$, $d^+\in\Nl^n$ and $d^-\in\Nl^m$. The normal form R-matrix $N$ with dimensions $d^+,d^-$ is     
     \begin{align}\label{eq:NormalForm}
	  N
	  \deq 
	  1_{d_1^+}\boxplus\ldots\boxplus 1_{d_n^+}\boxplus (-1_{d_1^-})\boxplus\ldots\boxplus(-1_{d_m^-})
	  \,.
     \end{align}
\end{definition}

Any R-matrix of the type \eqref{eq:NormalForm} will be called normal form R-matrix. Note that in view of Prop.~\ref{proposition:Boxplus-preserves-YBE}, $N$ is indeed an involutive R-matrix. We emphasize that $N$ is not simply a multiple of the identity: For example, $1_1\boxplus1_1=F$ is the flip of dimension 2. 

\begin{lemma}\label{lemma:ParametersOfNormalForms}
     Let $N\in\R_0$ be the normal form R-matrix with dimensions $d^+\in\Nl^n,d^-\in\Nl^m$. Then $N$ has dimension $d=\sum_id_i^++\sum_jd_j^-$, and the Thoma parameters of $\chi_N$ are
     \begin{align}\label{eq:ThomaParametersAndDimensions}
     	\alpha_i=\frac{d^+_i}{d},\quad i=1,\ldots,n,\qquad \beta_j=\frac{d^-_j}{d},\quad j=1,\ldots,m.
     \end{align}
\end{lemma}
\begin{proof}
     Recall that the identity
     $1\in\R_0(V)$ has $\alpha_1=1$ as its only non-vanishing Thoma parameter, and the negative identity 
     $-1\in\R_0(V)$ has $\beta_1=1$ as its only non-vanishing Thoma parameter, independently of the dimension of $V$. From this observation and the fact that $\boxplus$ adds dimensions, one can easily compute dimension and the Thoma parameters of $N$ \eqref{eq:NormalForm} by iterating Prop.~\ref{proposition:boxplus+thoma}~$ii)$, with the claimed result for $d$ and $\alpha=d^+/d$ and $\beta=d^-/d$.
\end{proof}

In Thm.~\ref{theorem:YB-parameters} we had proven that the Thoma parameters of every Yang-Baxter character lie in $\ThomaYB$. Lemma~\ref{lemma:ParametersOfNormalForms} now implies the converse, finishing the proof of Thm.~\ref{theorem:ThomaYB-Introduction} from the Introduction.

\begin{theorem}\label{theorem:ThomaYB1:1}
	The Yang-Baxter characters of $S_\infty$ are in one to one correspondence with $\ThomaYB$ (Def.~\ref{definition:TYB}) via Thoma's formula \eqref{eq:Thoma-Formula}.
\end{theorem}
\begin{proof}
	Let $(\alpha,\beta)\in\ThomaYB$. All that remains to be shown is that there is an R-matrix with these Thoma parameters. There exists $d\in\Nl$ such that all $d\alpha_i$, $d\beta_j$ are integer because the $\alpha_i$, $\beta_j$ are rational and finite in number (Def.~\ref{definition:TYB}). The character of the normal form R-matrix $N$ with dimensions $d_i^+=d\alpha_i$, $d_j^-=d\beta_j$ then has Thoma parameters $(\alpha,\beta)$ by Lemma~\ref{lemma:ParametersOfNormalForms} and the fact that the Thoma parameters sum to 1 (Cor.~\ref{corollary:YB->gamma+finite}).
\end{proof}

This result also justifies the notation $\ThomaYB $ as the Yang-Baxter simplex, consisting of all Thoma parameters of Yang-Baxter characters. Thoma's simplex~$\Thoma$, viewed as a subset of $[0,1]^\infty\times[0,1]^\infty$, where $[0,1]^\infty$ is equipped with the product topology, is a compact metrizable space. It is noteworthy to point out that $\ThomaYB\subset\Thoma$ is a dense subset, cf. \cite[Ch.~3]{BorodinOlshanski:2017}.

\medskip

At this stage, we know that for every R-matrix $R\in\R_0$, there exists a normal form R-matrix $N$ such that $\chi_R=\chi_N$. Furthermore, $N$ can be chosen in such a way that $R\sim N$, i.e. such that also the dimensions of $R$ and $N$ coincide. To see this, we just need to recall that the rescaled Thoma parameters $d\alpha_i$, $d\beta_i$ of $R$ are integers summing to $d$ (Thm.~\ref{theorem:YB-parameters}), so that the normal form R-matrix with dimensions $d_i^+=d\alpha_i$, $d_i^-=d\beta_i$ has the same character and the same dimension as $R$.

We briefly mention further properties of normal form R-matrices: Any normal form R-matrix $N$ commutes with the flip because of Lemma~\ref{lemma:boxplus}~$iii)$. Thus any involutive R-matrix is equivalent to an R-matrix which commutes with the flip, though this need not be true for an R-matrix not in normal form.

Furthermore, one can check that any normal form R-matrix $N$ satisfies
\begin{align}\label{eq:NptrN}
	N(1\ot\ptr N)N=\ptr N\ot1\,.
\end{align}

By Thm.~\ref{theorem:PartialTraces}, a normal form R-matrix $N$ (of dimension $d$) satisfies, as any involutive R-matrix, 
\begin{align}\label{eq:PartialTraceFormula-N}
	\chi_N(c_n)=d^{-n}\Tr_{V\tp{n}}(N_1\cdots N_{n-1})=d^{-n}\Tr_V((\ptr N)^{n-1}),
\end{align}
where $c_n$ is an $n$-cycle, $n\geq2$. With the exchange relation \eqref{eq:NptrN}, it is a matter of explicit calculation to prove \eqref{eq:PartialTraceFormula-N} directly for normal form R-matrices, without relying on subfactor theory.

\subsection{Parameterization by pairs of Young diagrams}\label{section:Young}

The correspondence in Thm.~\ref{theorem:ThomaYB1:1} classifies the family of Yang-Baxter {\em characters}, but it does not classify $\Rquot$ because the dimension of the base space is not recorded in the Thoma parameters. However, it is now easy to incorporate the dimension as well: Given $R\in\R_0$ with Thoma parameters $(\alpha,\beta)$ and dimension~$d$, we switch to the {\em rescaled Thoma parameters}
\begin{align}\label{eq:ab}
	a_i\deq d\alpha_i,\quad b_i\deq d\beta_i.
\end{align}
By Thm.~\ref{theorem:YB-parameters}, the $a_i,b_i$ are integers summing to $d$. We can therefore view $(a,b)$ as an ordered pair of integer partitions, or, equivalently, Young diagrams. Denoting the set of all Young diagrams (with an arbitrary number of boxes) by~$\Young$, we arrive at the following theorem, which in particular implies Thm.~\ref{theorem:Rquot=YxY-Introduction} from the Introduction.

\begin{theorem}\label{theorem:Young-Parameterization}\leavevmode
	\begin{enumerate}
		\item $\Rquot$ is in one to one correspondence with $\Young\times\Young\backslash\{(\emptyset,\emptyset)\}$ via mapping $[R]$ to the pair $(a,b)$ \eqref{eq:ab}. Classes of R-matrices of dimension~$d$ correspond to pairs of Young diagrams with $d$ boxes in total.
		\item Let $R\in\R_0$. The eigenvalues of $\ptr R$ lie in $\{\pm 1,\pm 2,\ldots,\pm d\}$ and for each eigenvalue $\la$, there exists $n_\la\in\Nl$ such that its multiplicity is $n_\la\cdot|\la|$. Define an integer partition $a$ as the ordered set of positive eigenvalues, in which $\lambda$ is repeated $n_\la$ times, and analogously for $b$ and the negative eigenvalues. Then $R$ corresponds to $(a,b)$ via the bijection in part~i).
	\end{enumerate}	
\end{theorem}
\begin{proof}
	$i)$ If $R,S\in\R_0$ are equivalent, they have the same dimension $d$ and the same Thoma parameters $(\alpha,\beta)$ and hence the same rescaled parameters \eqref{eq:ab}. Conversely, if $R,S\in\R_0$ have the same parameters $(a,b)$, they have the same dimension $d=\sum_i(a_i+b_i)$ and  therefore the same Thoma parameters, i.e. $R\sim S$. This also shows the claim about the dimension, and that all pairs of Young diagrams with the exception of $(\emptyset,\emptyset)$ occur.
	
	$ii)$	We may switch from $R$ to its normal form $N$ (with dimensions $d^\pm$), which has the same partial trace $\ptr N\cong\ptr R$. Repeated application of Lemma~\ref{lemma:boxplus}~$iv)$ shows $\ptr N=\bigoplus_i\ptr(1_{d_i^+})\oplus\bigoplus_i\ptr(-1_{d_i^-})$. But $\ptr(\pm1_{d_i^\pm})=\pm d_i^\pm \id_{i,\pm}$, where $\id_{i,\pm}$ is the identity matrix on $\Cl^{d_i^\pm}$. Hence the eigenvalues of $\ptr R$ are precisely the numbers $\pm d_i^\pm$. The eigenvalue $\pm d_i^\pm$ has multiplicity $n_{d_i^\pm}\cdot d_i^\pm$, where $n_{d_i^\pm}$ is the number of times that $d_i^\pm$ occurs in $d^\pm$.
	
	In view of \eqref{eq:ThomaParametersAndDimensions} and  \eqref{eq:ab}, the rescaled Thoma parameters of $N$ are exactly $a_i=d_i^+$, $b_i=d_i^-$. As $d_i^\pm$ occurs $n_{d_i^\pm}$ times in this list, the proof is finished.
\end{proof}

To illustrate the correspondence with ordered pairs of Young diagrams, let us list all normal forms of dimension two in terms of box sums and diagrams:
\begin{center}
	\begin{tabular}{|c|c|c|c|c|}
			\hline
			$1_2$ & $-1_2$ & $1_1\boxplus1_1$  & $-1_1\boxplus-1_1$ & $1_1\boxplus-1_1$
			\\
			\hline
                        $(\ydiagram{2},\emptyset)$&
                        $(\emptyset,\ydiagram{2})$&
                        $(\ydiagram{1,1},\emptyset)$&
                        $(\emptyset,\ydiagram{1,1})$&
                        $(\ydiagram{1},\ydiagram{1})$
			\\
			\hline
	\end{tabular}
\end{center}
From left to right, these R-matrices (in $\End(\Cl^4)$) are: 1) the identity, 2) the negative identity, 3) the flip, 4) equivalent to the negative flip, and 5)
\begin{align*}
	1_1\boxplus-1_1=\left(
			\begin{array}{cccc}
				\;1\\
				&&\;1\\
				&\;1\\
				&&&-1
			\end{array}
			\right)\,.
\end{align*}
As a higher dimensional example, consider $(\ydiagram{3,1},\ydiagram{2,2})$. This R-matrix has dimension~$8$ (the number of boxes) and Thoma parameters $\alpha=(\frac{3}{8},\frac{1}{8})$, $\beta=(\frac{1}{4},\frac{1}{4})$.

\bigskip

The rescaled Thoma parameters are also useful for describing the $\boxplus$ operation introduced in Sect.~\ref{section:NormalForms}. We have already seen that $\boxplus$ gives $\R_0$ the structure of an abelian semigroup and preserves equivalence, i.e., descends to the quotient $\Rquot$. Recalling the effect of $\boxplus$ on the level of Thoma parameters (Prop.~\ref{proposition:boxplus+thoma}~$ii)$), it becomes apparent that for the rescaled parameters, we have
\begin{align}
	(a,b)\boxplus(a',b')=(a\cup a',b\cup b'),
\end{align}
where $a\cup a'$ denotes the partition whose parts are the union of those of $a$ and $a'$.

\medskip

Another operation on $\R$ is the tensor product of R-matrices. For $R\in\R(V)$, $S\in\R(W)$, we define $R\boxtimes S\in\R(V\ot W)$ by
\begin{align}
	R\boxtimes S = F_2(R\ot S)F_2 : V\ot W\ot V\ot W\to V\ot W\ot V\ot W,
\end{align}
where $F_2$ exchanges the second and third tensor factors. It is evident that $\boxtimes$ preserves the Yang-Baxter equation and involutivity, i.e. it defines a product on $\R$ and $\R_0$.

\begin{lemma}\label{lemma:tensorproduct}
	Let $R,R'\in\R_0$ have rescaled Thoma parameters $(a,b)$ and $(a',b')$, respectively. Then the rescaled Thoma parameters of $R\,\boxtimes\,R'$ are the non-increasing arrangements of 
	\begin{align}
		\{\hat a_{ij}\}&=\{a_ia_j',\,b_ib_j'\},\nonumber\\
		\{\hat b_{ij}\}&=\{a_i b_j',\,b_i a_j'\}.
	\end{align}
\end{lemma}
\begin{proof}
	With $d,d'$ the dimensions of $R,R'$, we have on an $n$-cycle, $n\geq2$,
	\begin{align*}
		(d\cdot d')^{n}\chi_{R\boxtimes R'}(c_n)
		&=
		\Tr_{(V\ot W)\tp{n}}((R\boxtimes R')_1\cdots(R\boxtimes R')_{n-1})
		\\
		&=
		\Tr_{V\tp{n}}(R_1\cdots R_{n-1})
		\Tr_{W\tp{n}}(R'_1\cdots R'_{n-1})
		\\
		&=
		d^n\chi_R(c_n)\cdot (d')^n\chi_{R'}(c_n)
		\\
		&=
		\left(\sum_ia_i^n+(-1)^{n+1}\sum_jb_j^n\right)
		\left(\sum_k (a'_k)^n+(-1)^{n+1}\sum_l(b'_l)^n\right)
		\\
		&=
		\sum_{i,k}(a_ia'_k)^n+\sum_{j,l}(b_jb'_l)^n
		+(-1)^{n+1}\left(\sum_{i,l}(a_ib'_l)^n
		+\sum_{j,k}(b_ja_k')^n\right),
	\end{align*}
	and the claim follows.
\end{proof}

It follows that $\boxtimes$ defines an associative commutative product on $\Rquot$ for which the class $[1_1]=(\ydiagram{1},\emptyset)$ (consisting of the identity R-matrix in dimension $d=1$) is the unit, that is, $\Rquot$ has a second unital abelian semigroup structure.

From the description of $\boxtimes$ and $\boxplus$ in terms of the rescaled Thoma parameters, it is evident that they satisfy the distributive law
\begin{align}
	([R]\boxplus[S])\boxtimes[T]
	=
	([R]\boxtimes[T])\boxplus([S]\boxtimes[T])
	\,,
	\qquad R,S,T\in\R_0\,.
\end{align}
These operations give $\Rquot$ the structure of a semiring, sometimes also referred to as a rig (ring without negatives).
Additionally, the multiplication rules for rescaled Thoma parameters in
Lemma \ref{lemma:tensorproduct} can be generalized to a
\(\lambda\)-operation. This \(\lambda\)-operation is most easily described
using symmetric polynomials and we therefore postpone it until
Sect.~\ref{section:symmfunctions}. The consequences of the ring and \(\lambda\) structures of \(\Rquot\) will be an interesting topic of further study. For example, can the \(\lambda\) operation be directly interpreted in terms of \(R\) and its base space \(V\) without reference to rescaled Thoma parameters?

As yet another operation on $\R$, we briefly mention the cabling procedure known from the braid groups, applied to the Yang-Baxter equation by Wenzl \cite{Wenzl:1990}: Given any $p\in\Nl$, one can form cabling powers $R^{{\rm c}(p)}$, which lie in $\R$ (or $\R_0$) if $R$ does. We do not give details here because it turns out that $R^{{\rm c}(p)}\sim R^{\boxtimes p}$ for all $R\in\R_0$, $p\in\Nl$.

\section{Yang-Baxter representations}\label{section:YBReps}

Our basic Def.~\ref{Definition:R0-Equivalence} of equivalence of R-matrices refers only to the $S_n$-representations $\rho_R^{(n)}$. We have seen already that $R\sim S$ implies unitary equivalence of the GNS representations $\pi_\tau\circ\rho_R\cong\pi_\tau\circ\rho_S$ \eqref{eq:GNS-equivalence}. Now we investigate the implications of $R\sim S$ for the homomorphisms $\rho_R$, $\rho_S$.

\subsection{R-matrices and $K$-theory}
In this section we extend the previously defined $\rho_R$ to a ${}^*$-homomorphism of $C^*$-algebras, $\rho_R \colon C^*S_\infty \to \E^{\infty}$, where $\E^{\infty}$ is the $C^*$-algebraic counterpart of the algebra $\E$ from Section~\ref{section:subfactors}. On $K$-theory the map $\rho_R$ will induce a ring homomorphism $\rho_{R*} \colon K_0(C^*S_\infty) \to \Zl[\tfrac{1}{d}]$ with $d = \dim(V)$. The equivalence relation introduced in Def.~\ref{Definition:R0-Equivalence} will then translate into the approximate unitary equivalence of the corresponding ${}^*$-homomorphisms. In fact, when the invariant $\rho_{R*}$ is composed with the canonical inclusion $\Zl[\tfrac{1}{d}] \subset \Rl$ it is an indecomposable finite trace on $K_0(C^*S_\infty)$ in the sense of Kerov and Vershik and we recover \cite[Thm.~2.3]{VershikKerov:1983} from the Yang-Baxter equation. For the basic facts about UHF-algebras that we use we refer the reader to \cite{RoerdamStoermer:2002}.

Let $R \in \R_0(V)$, let $d = \dim(V)$ and denote the associated unitary representation of $S_n$ by $\rho_R^{(n)}$. We obtain the following sequence of ${}^*$-homomorphisms, which we will continue to denote $\rho^{(n)}_R$: 
\[
	\rho^{(n)}_R \colon C^*S_n = \Cl[S_n] \to \End(V^{\otimes n}) 
\]
Let $\E^{\infty}$ be the $C^*$-algebra obtained as the infinite tensor product of the algebras $\End(V)$, i.e.\ as the $C^*$-algebraic inductive limit
\[
	\E^{\infty} = \varinjlim_n \End(V^{\otimes n})
\]
taken over the maps sending $T$ to $T \otimes \id_V$. This is an infinite UHF-algebra with $K_0(\E^{\infty}) \cong \Zl[\tfrac{1}{d}]$ and $K_1(\E^{\infty}) \cong 0$. The algebra $\E^{\infty}$ has a unique trace that restricts to the normalized trace on $\End(V^{\otimes n})$. It induces an explicit isomorphism $\tau_* \colon K_0(\E^{\infty}) \to \Zl[\tfrac{1}{d}]$ as follows: Let $p,q \in M_N(\E^{\infty})$ be projections. Then 
\[
	\tau_*([p] - [q]) = (\Tr_{N} \otimes \tau)(p) - (\Tr_{N} \otimes \tau)(q)\ ,
\]
where $\Tr_{N} \otimes \tau \colon M_N(\E^{\infty}) = M_N(\Cl) \otimes \E^{\infty} \to \Cl$ is induced by the non-normalized trace $\Tr_N$ tensored with $\tau$. This is in fact a ring isomorphism. To understand the ring structure on $K_0(\E^{\infty})$, note that $\E^{\infty}$ is strongly self-absorbing \cite[Ex.~1.14]{TomsWinter:2007}. In particular, there is an isomorphism $\psi \colon \E^{\infty} \otimes \E^{\infty} \to \E^{\infty}$ and any two such isomorphisms are homotopic. Let $p_i \in M_{N_i}(\E^{\infty})$ for $i \in \{1,2\}$ be projections and let $[p_i] \in K_0(\E^{\infty})$ be the corresponding $K$-theory classes. Let $\psi' \colon M_{N_1}(\E^{\infty}) \otimes M_{N_2}(\E^{\infty}) \to M_{N_1N_2}(\E^{\infty})$ be the isomorphism induced by $\psi$. Then we have $[p_1] \cdot [p_2] = [\psi'(p_1 \otimes p_2)]$. It follows from the uniqueness of the normalized trace on $\E^{\infty}$ that 
\[
	(\Tr_{N_1} \otimes \tau) \otimes (\Tr_{N_2} \otimes \tau) = (\Tr_{N_1N_2} \otimes \tau) \circ \psi'\ ,
\]
which implies $\tau_*([p_1] \cdot [p_2]) = \tau_*([p_1]) \cdot \tau_*([p_2])$.

Note that $\E^{\infty} \subset \E$ with $\E$ as in Section~\ref{section:subfactors}. The inductive limit of the representations $\rho_R^{(n)} \colon C^*S_n \to \End(V^{\otimes n})$ provides us with a ${}^*$-homomorphism 
\begin{equation} \label{eqn:rho_R}
	\rho_R \colon C^*S_\infty \to \E^{\infty}\ .
\end{equation}

The $K$-theory of $C^*S_\infty$ was studied by Kerov and Vershik in
\cite{VershikKerov:1983}. In particular, they obtained that
$K_0(C^*S_\infty)$ is isomorphic to a quotient of the ring of
symmetric functions. As an abelian group it is therefore spanned by projections $p_{\lambda} \in K_0(C^*S_\infty)$, that are labeled by partitions $\lambda =\partn{\lambda_1, \dots, \lambda_k}$ of natural numbers $n \in \Nl$. The map $\rho_R$ induces a group homomorphism 
\begin{equation} \label{eqn:K0invariant}
	\rho_{R*} \colon K_0(C^*S_\infty) \to K_0(\E^{\infty})
\end{equation}
in $K$-theory. Using the ring isomorphism induced by the unique trace on $\E^{\infty}$ we will identify $K_0(\E^{\infty})$ with $\Zl[\tfrac{1}{d}]$. The following lemma shows that $\rho_{R*}$ remembers the equivalence class of $R$.
\begin{lemma} \label{lem:rho_R_multiplicities}
	Let $\lambda$ be a partition of $n \in \Nl$. We will identify $\lambda$ with the corresponding irreducible representation of $S_n$. On the projection $p_{\lambda} \in C^*S_\infty$ associated to $\lambda$ the value of $\rho_{R*}$ is given by 
	\[
		\rho_{R*}([p_{\lambda}])= \frac{1}{d^n} \langle \lambda, \rho^{(n)}_R \rangle\ ,
	\] 
	where $\langle \lambda, \mu \rangle$ denotes the multiplicity of the irreducible representation $\lambda$ in the representation $\mu$. 
\end{lemma}

\begin{proof} \label{pf:rho_R_multiplicities}
We have $\rho_{R*}([p_{\lambda}]) = \tau(\rho_R(p_{\lambda}))$. Let $\tau_n \colon \End(V^{\otimes n}) \to \Cl$ be the normalized trace. Since $p_{\lambda} \in C^*S_n \subset C^*S_\infty$ and the inclusion $\End(V^{\otimes n}) \to \E^{\infty}$ preserves the normalized trace, we obtain
\[
	\tau(\rho_R(p_{\lambda})) = \tau_n(\rho_R^{(n)}(p_{\lambda})) = \frac{1}{d^n}\,\Tr_{V\tp{n}}(\rho_R^{(n)}(p_{\lambda}))\,.
\]
Let $V_R = V^{\otimes n}$ be the representation space of $\rho_R^{(n)}$. The decomposition into its irreducible components gives
\[
	V_R \cong \bigoplus_{\mu \in \text{Irrep}(S_n)} \hom_{C^*S_n}(V_{\mu}, V_R) \otimes V_{\mu}\ ,
\]
where the action on the left is via $\rho_R^{(n)}$ and on the right acts only on the second tensor factor $V_{\mu}$ via $\mu$. Observe that $p_{\lambda}V_{\mu}$ is zero for $\lambda \neq \mu$ and $1$-dimensional for $\lambda = \mu$. Hence, 
\begin{align*}
	\Tr_{V\tp{n}}(\rho_R^{(n)}(p_{\lambda})) &= \dim(\rho_R^{(n)}(p_{\lambda})V_R) = \dim(\hom_{C^*S_n}(V_{\lambda},V_R) \otimes p_{\lambda}V_{\lambda}) \\
	&=  \dim(\hom_{C^*S_n}(V_{\lambda},V_R)) = \langle \lambda, \rho_R^{(n)} \rangle\ . \qedhere
\end{align*}
\end{proof}

From this we obtain two useful additional characterizations of the equivalence relation from Def.~\ref{Definition:R0-Equivalence}, one of them $K$-theoretic, the other one $C^*$-algebraic. For the second one we need the following equivalence relation \cite[Def.~1.1.15]{RoerdamStoermer:2002}:
\begin{definition} \label{def:approx_unitary_eq}
	Let $\varphi, \psi \colon A \to B$ be ${}^*$-homomorphisms between separable unital $C^*$-algebras $A$ and $B$. We call them \emph{approximately unitarily equivalent} if there is a sequence of unitaries $u_n \in B$ with the property that for all $a \in A$ we have 
	\[
		\lim_{n \to \infty} \lVert \varphi(a) - u_n\,\psi(a)\,u_n^* \rVert = 0\ .
	\]
	We denote this by $\varphi \approx_u \psi$.
\end{definition}  
\begin{theorem} \label{thm:equiv_relations}
Let $R,S \in \R_0(V)$. The following are equivalent:
\begin{enumerate}
	\item \label{it:our} $R \sim S$,
	\item \label{it:KT} $\rho_{R*} = \rho_{S*}$,
	\item \label{it:au} $\rho_R \approx_u \rho_S$.
\end{enumerate}
\end{theorem}

\begin{proof} \label{pf:equiv_relations}
	The equivalence of $i)$ and $ii)$ is a consequence of Lemma~\ref{lem:rho_R_multiplicities} and the fact that $\rho_R^{(n)}$ and $\rho_S^{(n)}$ are unitarily equivalent if and only if the multiplicities of their irreducible subrepresentations agree. 
	
	To see that $ii)$ and $iii)$ are equivalent, note that the $C^*$-algebras $C^*S_\infty$ and $\E^{\infty}$ are both AF-algebras. The statement then follows from \cite[Prop.~1.3.4]{RoerdamStoermer:2002}.
\end{proof}

The $K$-group $K_0(C^*S_\infty)$ is in fact a ring: Let $\lambda$ be a partition of $n \in \Nl$ and let $\mu$ be a partition of $m \in \Nl$. Denote by $p_{\lambda}, p_{\mu} \in C^*S_\infty$ the associated projections. Let $\iota_{n,m} \colon C^*S_n \otimes C^*S_m \to C^*S_{n + m}$ be the ${}^*$-homomorphism induced by the inclusion $S_n \times S_m \to S_{n+m}$, where $S_n$ permutes the first $n$ elements and $S_m$ the last $m$ elements. The product $[p_{\lambda}] \cdot [p_{\mu}]$ is then defined to be the class of the projection $\iota_{n,m}(p_{\lambda} \otimes p_{\mu})\in C^*S_{n+m}\subset C^*S_\infty$ in $K_0(C^*S_\infty)$. With respect to this ring structure we make the following observation:

\begin{proposition} \label{thm:rho_R_ring_hom}
Let $R \in \R_0(V)$. Then the associated $K$-theory invariant 
\[
	\rho_{R*} \colon K_0(C^*S_\infty) \to \Zl[\tfrac{1}{d}]
\]
is a ring homomorphism.
\end{proposition} 

\begin{proof} \label{pf:rho_R_ring_hom}
Let $\lambda, \mu$ be partitions of $n,m \in \Nl$ respectively. Let $p_{\lambda} \in C^*S_n, p_{\mu} \in C^*S_m$ be the corresponding projections. Since the representations $\rho_R^{(n)}$ arise from the same R-matrix, we have $\rho_R^{(n+m)} \circ \iota_{n,m} = \rho_R^{(n)} \otimes \rho_R^{(m)}$. Hence, we obtain
\[
\rho_{R*}([p_{\lambda}]\cdot [p_{\lambda}]) = [\rho_R^{(n+m)}\circ \iota_{n,m}(p_{\lambda} \otimes p_{\mu})] = [\rho_R^{(n)}(p_{\lambda}) \otimes \rho_R^{(m)}(p_{\mu})]
\]
and after application of the isomorphism $\tau_* \colon K_0(\E^{\infty}) \to \Zl[\tfrac{1}{d}]$ induced by the trace:
\begin{align*}
	\tau_*(\rho_{R*}([p_{\lambda}]\cdot [p_{\lambda}])) 
	&= \tau_{n+m}(\rho_R^{(n)}(p_{\lambda}) \otimes \rho_R^{(m)}(p_{\mu})) \\ 
	&= \tau_n(\rho_R^{(n)}(p_{\lambda}))\,\tau_m(\rho_R^{(m)}(p_{\mu}))= \tau_*(\rho_{R*}([p_{\lambda}])) \cdot \tau_*(\rho_{R*}([p_{\mu}]))
\end{align*}
where $\tau_r \colon \End(V^{\otimes r}) \to \Cl$ denotes the (normalized) trace on the matrix algebra and we used that the inclusion $\End(V^{\otimes r}) \subset \E^{\infty}$ is trace preserving.
\end{proof}

Finite traces on $K_0(C^*S_\infty)$ have been studied by Kerov and Vershik in \cite{VershikKerov:1983} and $\rho_{R*}$ can be seen as a refinement of such a trace taking values in $\Zl[\tfrac{1}{d}]$. In particular, we recover the multiplicativity proven in \cite[Thm.~2.3]{VershikKerov:1983} in Prop.~\ref{thm:rho_R_ring_hom}. From the point of view of $C^*$-algebras it is remarkable that we obtain a ring homomorphism on $K$-theory that is induced by a ${}^*$-homomorphism. 

\subsection{R-matrices, $K$-theory and symmetric functions}
\label{section:symmfunctions}

  Before discussing the connections between R-matrices, $K$-theory and symmetric functions in greater detail, we first collect some facts about symmetric functions to fix notation. We refer readers unfamiliar with symmetric functions to Macdonald's book \cite{Macdonald:1995}.
  
  The ring of symmetric functions, $\symm$, admits numerous free generators. Here, the most important are:
  \begin{enumerate}
  \item Elementary symmetric functions:
    \begin{displaymath}
      \elsym{k}=\sum_{1\leq i_1<i_2<\cdots<i_k} x_{i_1}x_{i_2}\cdots x_{i_k},\quad k\geq1.
    \end{displaymath}
  \item Complete symmetric functions:
    \begin{displaymath}
      \comsym{k}=\sum_{1\leq i_1\leq i_2\leq \cdots\leq i_k} x_{i_1}x_{i_2}\cdots x_{i_k},\quad k\geq1.
    \end{displaymath}
  \item Power sums:
    \begin{displaymath}
      \powsum{k}=\sum_{i\geq1}x_i^k.
    \end{displaymath}
  \end{enumerate}
  The ring of symmetric functions also admits many interesting bases usually indexed by partitions of integers. For example the above three sets of generators each define a basis by having the basis vector associated to a partition \(\lambda=\partn{\lambda_1,\lambda_2,\dots}\) be \(f_\lambda= f_{\lambda_1}f_{\lambda_2}\cdots\), where \(f\) is either \(\elsym{},\comsym{}\) or \(\powsum{}\) and one defines \(f_0=1\). An additional important basis is given by Schur functions \(\schur{\lambda}\), which in terms of elementary and complete symmetric functions are given by the following determinantal formulae:
  \begin{displaymath}
    \schur{\lambda}=\det\left(\comsym{\lambda_i-i+j}\right)_{1\leq i,j\leq n}
      =\det\left(\elsym{\lambda_i'-i+j}\right)_{1\leq i,j\leq m},
  \end{displaymath}
  where \(n\geq \ell(\lambda)\), \(\lambda'\) is the partition conjugate to \(\lambda\) and \(m\geq \ell(\lambda')\).

The ring of symmetric functions also admits a ring involution \(\omega:\symm\to\symm\) which on the generators and bases defined above acts as
\begin{align*}
  \omega(\elsym{k})=\comsym{k},\quad
  \omega(\comsym{k})=\elsym{k},\quad
  \omega(\powsum{k})=(-1)^{k+1}\powsum{k},\quad
  \omega(\schur{\lambda})=\schur{\lambda'}.
\end{align*}
Finally, we will also make use of the coproduct \(\Delta:\symm\to\symm\ot\symm\) which maps a symmetric function \(f(x)\) to the same function \(f(x,y)\) but with the alphabet of variables split into two alphabets.

The images of the sets of generators defined above are then
\begin{align*}
  \Delta(\elsym{k})&=\sum_{q=0}^k\felsym{q}{x}\felsym{k-q}{y},\ 
  \Delta(\comsym{k})=\sum_{q=0}^k\fcomsym{q}{x}\fcomsym{k-q}{y},\\
  \Delta(\powsum{k})&=\fpowsum{k}{x}+\fpowsum{k}{y}.
\end{align*}
The corresponding formulae for Schur functions are more involved but can be
derived from their determinantal expressions in terms of elementary or
complete symmetric functions.

\medskip

As mentioned at the end of Sect.~\ref{section:Young}, the elementary symmetric functions and their coproducts can be used to define a \(\lambda\)-operation on rescaled Thoma parameters \((a,b)\). 
Denote the \(\lambda^n\) operation on \((a,b)\) by
\(\lambda^n(a,b)=(\lambda^n a,\lambda^n b)\), where \(\lambda^n a\) and
\(\lambda^n b\) are the non-decreasing arrangements of
\begin{align*}
  \lambda^n a&=\{\text{monomial summands of } \felsym{n}{a,b} \text{ with even number of factors from }b\},\\
  \lambda^n b&=\{\text{monomial summands of } \felsym{n}{a,b} \text{ with odd number of factors from }b\}.
\end{align*}
For example if \((a,b)=(\partn{a_1,a_2},\partn{b_1})\),
\begin{align*}
  \felsym{0}{a,b}&=1,\quad
  \felsym{1}{a,b}=a_1+a_2+b_1,\quad
  \felsym{2}{a,b}=a_1a_2+a_1b_1+a_2b_1,\\
  \felsym{3}{a,b}&=a_1a_2b_1,\quad
  \felsym{n}{a,b}=0,\ n\geq4.
\end{align*}
Thus,
\begin{align*}
  \lambda^0(a,b)&=(\partn{1},\emptyset),\quad
  \lambda^1(a,b)=(\partn{a_1,a_2},\partn{b_1}),\quad
  \lambda^2(a,b)=(\partn{a_1a_2},\partn{a_1b_1,a_2b_1}),\\
  \lambda^3(a,b)&=(\emptyset,\partn{a_1a_2b_1}),\quad
  \lambda^n(a,b)=(\emptyset,\emptyset),\ n\geq4.
\end{align*}

\begin{lemma}
  The operation \(\lambda^n\) is a \(\lambda\)-operation on $(\Rquot\cup\{(\emptyset,\emptyset)\},\boxplus,\boxtimes)$.
\end{lemma}
\begin{proof}
  An alternate appellation for the ring of symmetric functions over the
  integers is the free \(\lambda\)-ring in one generator, or more precisely,
  the ring of symmetric functions over the integers is the ring underlying the
  free \(\lambda\)-ring in one generator. The \(\lambda\)-operation above is
  essentially that of the ring of symmetric functions. See for example, \cite[Chapt.~1.3]{Yau:2010}.
\end{proof}

Let $I \subset \symm$ be the ideal generated by $\elsym{1}-1$
and let $\qsymm = \symm/I$. Kerov and Vershik pointed out that the homomorphism $\theta \colon \qsymm \to K_0(C^*S_\infty)$ fixed by $\theta(\schur{\lambda}) = [p_{\lambda}]$ is in fact a ring isomorphism \cite{VershikKerov:1983}. Using this identification we can now completely determine the $K$-theory invariant $\rho_{R*}$ in terms of the Thoma parameters of $R$.

\begin{theorem} \label{thm:ptr_and_K}
Let $R \in \R_0(V)$ with Thoma parameters $(\alpha, \beta)$.  
Then we have 
\[
	\rho_{R*}(\theta(\elsym{k})) = \left[(1 \otimes \omega) \circ \Delta(\elsym{k})\right](\alpha,\beta)
\]
\end{theorem}
\begin{proof} \label{pf:ptr_and_K}
	The generating function $g_R$ associated to the trace $\varphi = \tau_* \circ \rho_{R*} \circ \theta \colon \qsymm \to \Zl[\tfrac{1}{d}] \subset \Rl$ is given by  
	\[
		g_R(z) = \sum_{l = 0}^{\infty} \varphi(\elsym{l})z^l
	\]
	as described in \cite[eq.~(11)]{VershikKerov:1983} and is related to the Thoma parameters $(\alpha,\beta)$ as follows \cite[eq.~(12)]{VershikKerov:1983} (note that in the case at hand, $\gamma = 0$ and $N = \max\{n,m\}$ in the notation of \cite{VershikKerov:1983}):
	\[
		g_R(z) = \prod_{i = 1}^N \frac{1 + \alpha_i z}{1 - \beta_i z}\ .
	\] 
	Hence, the statement follows from the following computation and comparison of coefficients with $g_R(z)$:
	\begin{align*}
		\sum_{l = 0}^{\infty}\, [(1 \otimes \omega) \circ \Delta(\elsym{l})](\alpha,\beta)\,z^l
		=\ & \sum_{l = 0}^{\infty} \sum_{i + j = l} \felsym{i}{\alpha_1, \dots, \alpha_n}\,\fcomsym{j}{\beta_1, \dots, \beta_m}z^l \\
		=\ & \left(\sum_{i = 0}^{\infty} \felsym{i}{\alpha_1, \dots, \alpha_n}z^i\right) \left(\sum_{j = 0}^{\infty} \fcomsym{j}{\beta_1, \dots, \beta_m}z^j\right) \\
		=\ & \prod_{i = 1}^n (1 + \alpha_i z) \prod_{j = 1}^m \frac{1}{1 - \beta_j z} = \prod_{i = 1}^N \frac{1 + \alpha_i z}{1 - \beta_i z}\ .\qedhere
	\end{align*}
\end{proof}

An immediate consequence of the above theorem is that applying \(\rho_{R*}(\theta(-))\) to a power
sum \(\powsum{\lambda}\) is the same as evaluating the class function
\(\chi_R\) at a group element of cycle shape \(\lambda\). Moreover,
Lemma \ref{lem:rho_R_multiplicities}, Theorem \ref{thm:ptr_and_K}
and the fact that \(\theta(\schur{\lambda})=[p_\lambda]\) can now be used to
easily derive explicit formulae for the multiplicities of irreducible
representations of \(S_n\) in \(\rho_R^{(n)}\).
\begin{proposition}\label{Proposition:Multiplicities}
  Let \(R\in \R_0\) with rescaled
  Thoma parameters \((a,b)\), then the multiplicity of the \(S_n\)
  representation associated to a partition \(\lambda\) of \(n\) in
  \(\rho_R^{(n)}\) is
  \begin{equation}\label{eqn:multiplicities}
    \langle \lambda,
    \rho_R^{(n)}\rangle=\left[(1\otimes\omega)\circ\Delta(\schur{\lambda})\right](a,b). 
  \end{equation}
  Further let \(\ell(a),\ell(b)\) be the respective lengths of \(a\) and \(b\). Then 
  \(\langle \lambda, \rho_R^{(n)}\rangle=0\) if and only if the Young
  diagram of \(\lambda\) contains a rectangle of height $\ell(a)+1$ and width
  $\ell(b)+1$. If \(\lambda\) contains a rectangle of height \(\ell(a)\) and width
  \(\ell(b)\) (but not of respective height and width \(\ell(a)+1,\ell(b)+1\)), then
  \begin{displaymath}
    \langle \lambda,
    \rho_R^{(n)}\rangle=\fschur{\mu}{a}\fschur{\nu}{b}\prod_{i=1}^{\ell(a)}\prod_{j=1}^{\ell(b)}(a_i+b_j),
  \end{displaymath}
  where \(\mu,\nu\) are the partitions whose parts are \(\mu_i=\lambda_i-\ell(b),\
  i=1,\dots,\ell(b)\) and \(\nu_j=\lambda'_j-\ell(a),\ j=1,\dots,\ell(b)\).
\end{proposition}
\begin{proof}
  Let \(d\) be the dimension of \([R]\) and let
  \((\partn{a_1/d,a_2/d,\dots},\partn{b_1/d,b_2/d,\dots})\) be the
  associated Thoma parameters.
  The identity \eqref{eqn:multiplicities} follows by direct computation:
  \begin{align*}
    \langle \lambda, \rho_R^{(n)}\rangle&= d^n \rho_{R\ast}([p_\lambda])
    = d^n\rho_{R\ast}(\theta(\schur{\lambda}))\\
    &=d^n\left[(1\otimes \omega)\circ
    \Delta(\schur{\lambda})\right](a_1/d,a_2/d,\dots,b_1/d,b_2/d,\dots)\\
    &=\left[(1\otimes\omega)\circ\Delta(\schur{\lambda})\right](a,b).
  \end{align*}
  The remainder of the proposition is just Example 23 of Section 3 and Example~23 of Section 5 in \cite{Macdonald:1995}.
\end{proof}

The conditions for the vanishing of \(\langle \lambda,\rho_R^{(n)}\rangle\) were previously observed in \cite[Thm.~III.6.5]{Wassermann:1981} and in \cite[Thm.~6.9]{DoplicherHaagRoberts:1971}. An example of the multiplicities of irreducible \(S_n\) representations computed using Prop.~\ref{Proposition:Multiplicities} is given in
Fig.~\ref{fig:younglattice}.

  \begin{figure}[htp]
    \centering
    \scalebox{0.7}{
    \ytableausetup{boxsize=.9em}
    \begin{tikzpicture}[scale=0.8]
	\draw (0,6) node(s1) {\ytableaushort{{\scriptstyle 3}}*{1}};
	\draw (-1.5,4.5) node(s2) {\ytableaushort{{\scriptstyle 2}}*{2}};
	\draw (1.5,4.5) node(s11) {\ytableaushort{{\scriptstyle 7}}*{1,1}};
	\draw (-3,2.3) node(s3) {\ydiagram{3}*[*(lightgray)]{3}};
	\draw (0,2.3) node(s21) {\ytableaushort{{\scriptstyle 6}}*{2,1}};
	\draw (3,2.3) node(s111) {\ytableaushort{{\scriptstyle 15}}*{1,1,1}};
	\draw (-4.5,0) node(s4) {\ydiagram{4}*[*(lightgray)]{3}};
	\draw (-1.9,0) node(s31) {\ydiagram{3,1}*[*(lightgray)]{3}};
	\draw (0,0) node(s22) {\ytableaushort{{\scriptstyle 4}}*{2,2}};
	\draw (1.9,0) node(s211) {\ytableaushort{{\scriptstyle 14}}*{2,1,1}};
	\draw (4.5,0) node(s1111) {\ytableaushort{{\scriptstyle 31}}*{1,1,1,1}};
        \node at (-7.5,-3) (s5) {\ydiagram{5}*[*(lightgray)]{3}};
        \node at (-4.5,-3) (s41) {\ydiagram{4,1}*[*(lightgray)]{3}};
        \node at (-2.1,-3) (s32) {\ydiagram{3,2}*[*(lightgray)]{3}};
        \node at (0,-3) (s311) {\ydiagram{3,1,1}*[*(lightgray)]{3}};
        \node at (2.1,-3) (s221) {\ytableaushort{{\scriptstyle 12}}*{2,2,1}};
        \node at (4.5,-3) (s2111) {\ytableaushort{{\scriptstyle 30}}*{2,1,1,1}};
        \node at (7.5,-3) (s11111) {\ytableaushort{{\scriptstyle 63}}*{1,1,1,1,1}};

	\draw (s1) -- (s2);
	\draw (s1) -- (s11);
	\draw (s2) -- (s3);
	\draw (s2) -- (s21);
	\draw (s11) -- (s21);
	\draw (s11) -- (s111);
	\draw (s3) -- (s4);
	\draw (s3) -- (s31);
	\draw (s21) -- (s31);
	\draw (s21) -- (s22);
	\draw (s21) -- (s211);
	\draw (s111) -- (s211);
	\draw (s111) -- (s1111);
        \draw (s4) -- (s5);
        \draw (s4) -- (s41);
        \draw (s31) -- (s41);
        \draw (s31) -- (s32);
        \draw (s31) -- (s311);
        \draw (s22) -- (s32); 
        \draw (s22) -- (s221);
        \draw (s211) -- (s311);
        \draw (s211) -- (s221);
        \draw (s211) -- (s2111);
        \draw (s1111) -- (s11111);
        \draw (s1111) -- (s2111);
      \end{tikzpicture}
      }
      \ytableausetup{boxsize=0.5em,centertableaux}
 
        \caption{The Young lattice for the class of R-matrices associated to
          the pair $(\emptyset,\ydiagram{2,1})$. 
          Since the length of the first
          partition is 0 and that of the second is 2, any diagram containing a
          rectangle of height 1 and width 3 gives multiplicity 0.
          The first irreducible representation whose corresponding
          partition contains such a rectangle is the trivial representation of
          \(S_3\).
          For the
          remaining diagrams the multiplicities are given in the first
          box.}
    \label{fig:younglattice}
  \end{figure}

We end this section by comparing our classification of Yang-Baxter characters in terms of Thoma parameters with an alternative approach based on Hilbert-Poincar\'e series \cite{Gurevich:1991, Davydov:2000}. Writing $[1^n]=[1,\ldots,1]$ and $[n]$ for the partitions indexing the alternating and trivial representation of $S_n$, respectively, these series are defined as
\begin{align}
	H_R^-(z)
	:=
	\sum_{n=0}^\infty \langle[1^n],\rho_R^{(n)}\rangle\,z^n 
	\,,\qquad
	H_R^+(z)
	:=
	\sum_{n=0}^\infty \langle[n],\rho_R^{(n)}\rangle\,z^n \,.
\end{align}
Using a categorical description of R-matrices, it was shown in \cite[Thm.~1]{Davydov:2000} (for the case of Hecke algebra representations, see \cite{GPS}) that $H_R^-$ has to have the form 
	\begin{align} \label{eq:HR-}
		H_R^-(z) = \prod_{i=1}^N \frac{1 + a_i'z}{1 - b_i'z}
	\end{align}
	for suitable real positive values $a_i', b_i'$, $i \in \{1, \dots,
        N\}$, where we allow $a_i' = 0$ or $b_i' = 0$ to obtain sequences of
        equal length. Since $\schur{[1^n]}=\elsym{n}$, Thm.~\ref{thm:ptr_and_K} and Prop.~\ref{Proposition:Multiplicities} yield
	\begin{align}\label{eq:HilbertSeries}
		H_R^-(z) = \sum_{n = 0}^{\infty}\left[ (1 \otimes \omega) \circ \Delta(\elsym{n}) \right](a,b)\,z^n = g_R(d\cdot z) = \prod_{i=1}^N \frac{1 + a_i z}{1 - b_i z}\ .
	\end{align}
	Thus, we may pick $a_i' = a_i$ and $b_i' = b_i$ in
        \eqref{eq:HR-}. This allows us to extend the statement of
        \cite[Thm.~1]{Davydov:2000}: Not only is the Hilbert series a rational
        function, but its parameters $a_i$ and $b_i$ are non-negative
        integers. By a similar argument using the facts 
        $\schur{[n]}=\comsym{n}$ and $\omega(\comsym{n}) = \elsym{n}$ we obtain
	\begin{align}\label{eq:HilbertSeriesTriv}
		H_R^+(z) = \sum_{n = 0}^{\infty}\left[ (1 \otimes \omega) \circ \Delta(\elsym{n}) \right](b,a)\,z^n = \prod_{i=1}^N \frac{1 + b_i z}{1 - a_i z}\ .
	\end{align}
The following corollary collects some well-known results \cite{Gurevich:1991, Davydov:2000} about
Hilbert-Poincar\'e series which in our setting follow as immediate consequences of \eqref{eq:HilbertSeries}, \eqref{eq:HilbertSeriesTriv}, Prop.~\ref{proposition:boxplus+thoma} and Lemma~\ref{lemma:tensorproduct}.

\begin{corollary}
	Let $R,S\in\R_0$.
	\begin{enumerate}
		\item $H^+_R(z)\cdot H^-_R(-z)=1$.\\[-2mm]
		\item $H_{-R}^\pm(z)=H^\mp_R(z)$.\\[-2mm]
		\item $H^\pm_{R\boxplus S}(z)=H^\pm_R(z)\cdot H_S^\pm(z)$.
	\end{enumerate}
\end{corollary}

\subsection{R-matrices and product states}\label{section:ProductStates}
	
Our Yang-Baxter characters are defined by composing the homomorphism $\rho_R:\Cl[S_\infty]\to\E_0$ with the unique tracial product state $\tau=\bigotimes_{n\geq1}\frac{\Tr_V}{d}$ of $\E_0$ \eqref{eq:DefChiR}. In this section, we briefly discuss how this construction extends to a much larger class of extremal characters when we change $\tau$ to a different product state on~$\E_0$: All extremal characters with Thoma parameters summing to 1 can, together with their GNS representations, be expressed in terms of R-matrices and product states. The essential difference to our Yang-Baxter setting is that instead of the canonical trace $\tau$, a different product state is used.

Let $R\in\R_0(V)$ and $Z\in\End(V)$ such that $[R,Z\ot Z]=0$. We consider the product state $\om_Z\deq \bigotimes_{n\geq1}\Tr_V(Z\,\cdot\,)$ on $\E_0$, which is tracial in restriction to $\rho_R(\Cl[S_\infty])$. Composed with $\rho_R$, we thus get a character $\om_Z^R\deq \om_Z\circ\rho_R$ of $S_\infty$. As in Prop.~\ref{proposition:chiR-extremal}, one shows that $\om_Z^R$ is extremal.

We want to show that any extremal $S_\infty$-character with Thoma parameters $(\alpha,\beta)$ satisfying $\sum_i(\alpha_i+\beta_i)=1$ is of this form. To this end, consider two Hilbert spaces $V_1$, $V_2$, such that $\dim V_1$ and $\dim V_2$ equal the number of non-vanishing $\alpha$'s and $\beta$'s, respectively (which might be countably infinite). Let us fix orthonormal bases  $\{e_i\}_i$ and $\{f_j\}_j$ of $V_1$ and $V_2$, and trace class operators $A\in\B(V_1)$, $B\in\B(V_2)$ defined by $Ae_i=\alpha_i\cdot e_i$, $Bf_j=\beta_j\cdot f_j$.

\begin{lemma}
	In the notation introduced above, consider the Hilbert space $W\deq V_1\oplus V_2$ and the (now possibly infinite dimensional) R-matrix $F\boxplus-F\in\R_0(W)$. Then 
	\begin{align}
		\om^{F\boxplus-F}_{A\oplus B}\deq \bigotimes_{n\geq1}\Tr_W((A\oplus B)\,\cdot\,)\circ\rho_{F\boxplus-F}:S_\infty\to\Cl
	\end{align}
	is an extremal character of $S_\infty$. Its $\alpha$-parameters are the eigenvalues of $A$, and its $\beta$-parameters are the eigenvalues of $B$.
\end{lemma}
\begin{proof}
	In view of the simple structure of $F\boxplus-F$, it is easy to see that this operator commutes with $Z\ot Z$, where $Z\deq A\oplus B$. Thus $\om_{A\oplus B}^{F\boxplus-F}$ is indeed an extremal character, and it remains to compute its Thoma parameters. Using the orthogonality $V_1\perp V_2$ and the direct sum structure of $Z=A\oplus B$, one shows in close analogy to Prop.~\ref{proposition:boxplus+thoma} (see, in particular, \eqref{eq:Rhat-traces}), that for any $n$-cycle,
	\begin{align}
		\om_{A\oplus B}^{F\boxplus-F}(c_n)
		&=
		\om_A^F(c_n)+\om_B^{-F}(c_n)\,.
	\end{align}
	Furthermore, 
	\begin{align*}
		\om_A^F(c_n)
		&=
		\sum_{i_1,\ldots,i_n}
		\alpha_{i_1}\cdots\alpha_{i_n}\,\langle e_{i_1}\ot\ldots\ot e_{i_n},\,e_{i_2}\ot\ldots\ot e_{i_n}\ot e_{i_1}\rangle
		=
		\sum_i\alpha_i^n,\\
		\om_B^{-F}(c_n)
		&=
		(-1)^{n+1}\sum_{j_1,\ldots,j_n}
		\beta_{j_1}\cdots\beta_{j_n}\,\langle f_{j_1}\ot\ldots\ot f_{j_n},\,f_{j_2}\ot\ldots\ot f_{j_n}\ot f_{j_1}\rangle\\
		&=
		(-1)^{n+1}\sum_j\beta_j^n\,.
	\end{align*}
	These two terms sum to the value of the extremal character with Thoma parameters $(\alpha,\beta)$.
\end{proof}

We next describe the GNS representation of $\om_{A\oplus B}^{F\boxplus-F}$, which turns out to be closely related to R-matrices as well. In the notation introduced above, let $V\deq W\ot W$ and 
\begin{align}\label{eq:Rbt}
	R\deq (F\boxplus-F)\boxtimes1\in\R_0(V)\,,
\end{align}
where $F\in\R_0(V_1)$, $-F\in\R_0(V_2)$, and $1\in\R_0(W)$. In $V$, we fix the unit vector
\begin{align}
	\xi\deq \sum_i\sqrt{\alpha_i}e_i\ot e_i+\sum_j\sqrt{\beta_j}f_j\ot f_j\in W\ot W=V.
\end{align}
This vector determines inclusions $V\tp{n}\to V\tp{(n+1)}$ by tensoring with $\xi$ from the right, and we denote the corresponding inductive limit Hilbert space $\bigotimes_{n\geq1}^\xi V$.

\begin{proposition}
	Let $\chi$ be an extremal $S_\infty$-character with Thoma parameters $(\alpha,\beta)$ satisfying $\sum_i(\alpha_i+\beta_i)=1$. Then the GNS data $(\pi_\chi,\Hil_\chi,\Om_\chi)$ can be described in terms of the previously introduced $V,R$, and $\xi$ as
	\begin{align}
		\Om_\chi=\bigotimes_{n\geq1}\xi,\qquad
		\Hil_\chi=\overline{\rho_R(\Cl[S_\infty])\Om_\chi},\qquad  \pi_\chi=\rho_R\,.
	\end{align}
\end{proposition}
\begin{proof}
	Observe that for $w_1,w_2,w_3,w_4\in W$, the R-matrix \eqref{eq:Rbt} acts according to
	\begin{align}
		R(w_1\ot w_2\ot w_3\ot w_4) 
		= 
		\pm\,w_3\ot w_2\ot w_1\ot w_4\,,
	\end{align}
	where the sign is negative if both $w_1$ and $w_3$ lie in $V_2$, and positive otherwise. 
	
	Let $w_1,\ldots, w_n,u_1,\ldots,u_n\in W$, and $\sigma\in S_n$. Then this action of $R$ implies
	\begin{align}
		\rho_R(\sigma)\bigotimes_{k=1}^n (w_k\ot u_k)
		=
		\pm
		\bigotimes_{k=1}^n(w_{\sigma^{-1}(k)}\ot u_k),
	\end{align}
	with the sign depending on the number of vectors $w_k$ lying in $V_2$. From here one verifies 
	\begin{align}
		\langle\bigotimes_{n\geq1}\xi,\rho_R(\sigma)\bigotimes_{n\geq1}\xi\rangle
		=
		\chi(\sigma)
	\end{align}
	by following \cite[Prop.~10.5, Prop.~10.6]{BorodinOlshanski:2017}.
\end{proof}

The representation in the above proposition is known (see the original literature \cite{Olshanski:1990,Wassermann:1981} or the monograph \cite{BorodinOlshanski:2017}, where also the relation to spherical representations of $S_\infty$ is discussed), but takes a particularly simple form in terms of our operations $\boxplus$ and $\boxtimes$.

\section{Examples}\label{section:examples}

In this section, we discuss two special classes of involutive R-matrices. 

\subsection{R-matrices of diagonal type}

As a simple class of examples which exist in any dimension, we consider involutive R-matrices of {\em diagonal type}. An R-matrix $R\in\R_0$ is said to be diagonal if it is of the form $R=DF$, with $F$ the flip, and for some orthonormal basis $\{e_i\}_i$ of $V$, the matrix $D\in\End(V\ot V)$ is diagonal in the corresponding tensor basis, i.e.
\begin{align}
	D(e_i\ot e_j)=\la_{ij}\,e_i\ot e_j\,,\qquad i,j=1,\ldots,d\,,
\end{align}
where $\la_{ij}\in\Cl$. It is easy to check that such $R$ solve the Yang-Baxter equation. An R-matrix is said to be of diagonal type if it is equivalent to a diagonal one. 

The R-matrix $R=DF$, $R(e_i\ot e_j)=\la_{ji}\,e_j\ot e_i$, is unitary and involutive if and only if
\begin{align}\label{eq:lambda}
	|\la_{ij}|=1\,,\qquad \la_{ji}=\la_{ij}^{-1}\,,\qquad i,j=1,\ldots,d.
\end{align}
In particular, we have $\la_{ii}=\pm1$ for each $i\in\{1,\ldots,d\}$, and we introduce the parameter $\ell\in\{0,\ldots,d\}$ as the number of $\la_{ii}$'s that are equal to $+1$. This parameter is uniquely fixed by the {\em rank} $r$ of $R$, which is defined as the multiplicity of the eigenvalue $+1$ of $R$. In fact, the trace of $R$ is 
\begin{align}\label{eq:Tr-r-vs-l}
	2\ell-d=\sum_{i=1}^d\la_{ii}=\Tr(R)=2r-d^2\,.
\end{align}
As $\ell$ ranges over $\{0,\ldots,d\}$, the rank $r$ ranges over
\begin{align}\label{eq:r-range-diagonal}
	\frac{1}{2}d(d-1)\leq r\leq\frac{1}{2}d(d+1)\,.
\end{align}
Thus diagonal involutive R-matrices of dimension $d$ and rank $r$ exist if and only if \eqref{eq:r-range-diagonal} is satisfied.

\begin{proposition}\label{proposition:diagonal}\leavevmode
	\begin{enumerate}
		\item Let $R\in\R_0$ be of diagonal type, with dimension $d$, rank $r$, and $\ell\deq r-\frac{1}{2}d(d-1)$. Then
		\begin{align}
			R\sim \bigboxplus_{i=1}^\ell1_1\boxplus\bigboxplus_{j=1}^{d-\ell}(-1_1)\,,
		\end{align}
		and the non-vanishing Thoma parameters of $R$ are 
		\begin{align}
		    \alpha_1=\alpha_2=\ldots=\alpha_{\ell}=\beta_1=\ldots=\beta_{d-\ell}=d^{-1}.
		\end{align}
		\item Any two involutive R-matrices of diagonal type with the same dimension and rank are equivalent.
	\end{enumerate}
\end{proposition}
\begin{proof}
	$i)$ In the basis defining $D$, one has 
	\begin{align*}
		\langle e_i,\,\ptr(R)\,e_j\rangle
		=
		\sum_k\langle e_i\ot e_k,\,D\,e_k\ot e_j\rangle
		=
		\la_{ii}\,\delta_{ij},
	\end{align*}
	which shows $\ptr(R)=\id_{\Cl^\ell}\oplus(-\id_{\Cl^{d-\ell}})$. The claim now follows from Thm.~\ref{theorem:PartialTraces}~$ii)$ and Thm.~\ref{theorem:Young-Parameterization}~$ii)$.
	
	$ii)$ The character depends only on $d$ and $\ell$, and the rank $r$ determines $\ell$ uniquely.
\end{proof}

In terms of diagrams, diagonal R-matrices have the form
\begin{align}
	\left(\ydiagram{1,1,1,1},\ydiagram{1,1,1}\right),\quad 
\end{align}
with $\ell$ boxes in the left and $d-\ell$ boxes in the right column. Yang-Baxter characters of diagonal R-matrices play a significant role in the analysis of the statistics of superselection sectors in quantum field theory \cite[Prop.~6.10]{DoplicherHaagRoberts:1971}. They also feature prominently in the context of tensor categories as the skew-invertible unitary R-matrices \cite{Lyubashenko:1987,Gurevich:1991}.

\subsection{Temperley-Lieb R-matrices}

As a second class of examples, we consider solutions coming from representations of the Temperley-Lieb algebra \cite{TemperleyLieb:1971_2}. Given an involutive R-matrix $R\in\R_0$, we denote its spectral projection onto eigenvalue $+1$ by $P$, i.e. $P=\frac{1}{2}(R+1)$. One computes
\begin{align}\label{eq:YBE-P}
	\frac{1}{8}\left( R_1R_2R_1-R_2R_1R_2\right)
	=
	\left(P_1P_2P_1-\frac{1}{4}P_1\right)-\left(P_2P_1P_2-\frac{1}{4}P_2\right),
\end{align}
and this vanishes by the Yang-Baxter equation. If both terms on the right hand side vanish individually,
\begin{align}\label{eq:P-TL}
	P_1P_2P_1=\frac{1}{4}P_1\,,\qquad P_2P_1P_2=\frac{1}{4}P_2\,,
\end{align}
then $R$ is said to be of {\em Temperley-Lieb type}. This terminology is justified by the close relation of \eqref{eq:P-TL} to the defining relations of the Temperley-Lieb algebra: Recall that given $\delta>0$, the Temperley-Lieb algebra $\TL(\delta)$ is the unital ${}^*$-algebra over $\Cl$ with generators $T_k$, $k\in\Nl$, and the relations
\begin{align*}
	T_k^2
	&=
	\delta\,T_k,\quad T_k^*=T_k,\\
	T_kT_m
	&=
	T_mT_k\,,\quad |k-m|\geq2\,,\\
	T_kT_mT_k
	&=
	T_k\,,\qquad\; |k-m|=1\,.
\end{align*}
Given an orthogonal projection $P\in\End(V\ot V)$ satisfying \eqref{eq:P-TL}, setting $T_k\deq 2\,P_k$, $k=1,\ldots,n-1$, defines a representation of the Temperley-Lieb algebra $\TL(\delta)$ with $\delta=2$.

Our equivalence relation $\sim$ preserves the property of being of Temperley-Lieb type, as follows from the lemma below.

\begin{lemma}\label{lemma:TL} 
	An R-matrix $R\in\R_0$ is of Temperley-Lieb type if and only if $\rho_R^{(3)}$ does not contain the trivial representation of $S_3$.
\end{lemma}
\begin{proof}
	Let $p_3\in\Cl[S_3]$ be the projection given by the trivial representation of $S_3$, represented as
     \begin{align}\label{eq:E3}
	  \rho_R(p_3)
	  =
	  \frac{1}{6}\left(R_1R_2R_1+R_1R_2+R_2R_1+R_1+R_2+1\right)\,.
     \end{align}
     Clearly, $\rho_R^{(3)}$ does not contain the trivial representation if and only if $\rho_R(p_3)=0$. Inserting $R=2P-1$ into \eqref{eq:E3} gives by straightforward calculation
     \begin{align*}
	  \rho_R(p_3)=\frac{4}{3}\left(P_1P_2P_1-\frac{1}{4}P_1\right)=\frac{4}{3}\left(P_2P_1P_2-\frac{1}{4}P_2\right).
     \end{align*}
     Thus $\rho_R(p_3)=0$ is equivalent to the Temperley-Lieb relations \eqref{eq:P-TL}.
\end{proof}

The example displayed in Fig.~\ref{fig:younglattice} describes a Temperley-Lieb type
R-matrix because the trivial representation of \(S_3\) does not appear in \(\rho_R^{(3)}\).

\begin{proposition}\label{proposition:TL-Rs}\leavevmode
	\begin{enumerate}
		\item Yang-Baxter representations of the Temperley-Lieb algebra $\TL(2)$ with rank~$r$ and dimension $d$ exist if and only if 
	\begin{align}\label{eq:quad}
		d^2-4r=k^2
	\end{align}
	for some $k\in\Nl_0$. Two such representations are equivalent if and only if they have the same dimension and rank.
		\item Let $R\in\R_0$ be an involutive R-matrix of Temperley-Lieb type with dimension $d$ and rank $r$. Then its non-vanishing Thoma parameters are
		\begin{align}\label{eq:ab-TL}
			\beta_1=\frac{1}{2}\left(1+\sqrt{1-\frac{4r}{d^2}}\right)\,,\quad \beta_2=\frac{1}{2}\left(1-\sqrt{1-\frac{4r}{d^2}}\right)\,.
		\end{align}
	\end{enumerate}
\end{proposition}
\begin{proof}
	We first show part $ii)$. Let $R\in\R_0$ have Thoma parameters $(\alpha,\beta)$. According to the preceding lemma, the trivial representation of $S_3$, corresponding to the Young diagram $\ydiagram{3}$, does not occur in $\rho_R^{(3)}$ if and only if $R$ is of Temperley-Lieb type. We can thus conclude from Prop.~\ref{Proposition:Multiplicities} that (equivalence classes of) Temperley-Lieb R-matrices are in one to one correspondence with those $(\alpha,\beta)\in\ThomaYB$ that have at most $\beta_1,\beta_2$ as their only non-vanishing entries.
	
	We have $\beta_1+\beta_2=1$, and on a two-cycle, we get
	\begin{align}
		\chi_R(c_2)
		=
		-\beta_1^2-\beta_2^2
		=
		\frac{\Tr_{V\ot V}(R)}{d^2}
		=
		\frac{2r-d^2}{d^2}
		.
	\end{align}
	Solving the resulting quadratic equation proves \eqref{eq:ab-TL}.
	
	$i)$ A Yang-Baxter representation of $\T(2)$ of dimension $d$ and rank $r$ exists if and only if a Temperley-Lieb R-matrix with the same parameters exists. Since the rescaled Thoma parameters are integers, we know that $k\deq d(\beta_1-\beta_2)$ is an integer. In view of \eqref{eq:ab-TL}, $k=\sqrt{d^2-4r}$. This shows that \eqref{eq:quad} is necessary for the existence of a representation with dimension $d$ and rank $r$.
	
	Conversely, if \eqref{eq:quad} holds for some $k\in\Nl_0$, then \eqref{eq:ab-TL} defines a Temperley-Lieb R-matrix with dimension $d$ and rank $r$. In terms of diagrams, 
	\begin{align}
		R=\left(\emptyset,\ydiagram{7,2}\right), 
	\end{align}
	consisting of $d$ boxes distributed over (one or) two rows on the right, with the difference in row lengths equal to $k$.
	
	The last statement follows because the Thoma parameters \eqref{eq:ab-TL} depend only on $d$ and~$r$.
\end{proof}

We conclude this discussion of Temperley-Lieb \(R\)-matrices with the following remarks.
	\begin{enumerate}
		\item Yang-Baxter representations of the Temperley-Lieb algebra with general loop parameter $\delta$, i.e., representations in which the tensor structure $T_k=1\tp{(k-1)}\ot T\ot1\ot\ldots$ is required for the generators of $\TL(\delta)$, have recently been studied by Bytsko. He found various inequalities between $\delta$, the dimension $d$, and the rank $r={\rm Tr}_{V\ot V}(P)$ that are necessary for such representations to exist \cite{Bytsko:2015,Bytsko:2015_3}. 
		
		Our results give a necessary and sufficient condition on $d$ and $r$ for Yang-Baxter representations of $\TL(2)$ with these parameters to exist, and classify such representations up to equivalence. 
		
		\item In a more general setting of Hecke algebra representations, where the eigenvalues of $R$ are $-1$ and $\lambda$, with $\lambda$ not a non-trivial root of unity, Gurevich gave a classification of all Temperley-Lieb R-matrices for the special case of rank\footnote{Note that the term ``rank'' has a different meaning in \cite{Gurevich:1991}.} $r=1$~\cite{Gurevich:1991}. These rank $1$ solutions can be used to construct Temperley-Lieb R-matrices in any dimension, which are, however, typically not normal, 
that is, $P$ is not selfadjoint, even in the case of $\lambda=1$.

		The subclass of all those Temperley-Lieb R-matrices which are involutive and unitary, i.e. have eigenvalue $\lambda=1$ and selfadjoint spectral projection $P$, are precisely those which we consider here. Prop.~\ref{proposition:TL-Rs} implies in particular that in this subclass, the assumption $r=1$ is very restrictive and only realized by a single equivalence class given by dimension $d=2$ and $R\sim -F$.
                \item The Temperley-Lieb R-matrices in the above examples have non-vanishing $\beta$-parameters (instead of $\alpha$'s) because we imposed the Temperley-Lieb relation on the spectral projection onto eigenvalue $+1$. If we used the spectral projection onto eigenvalue $-1$ instead, $\alpha$ and $\beta$ would be exchanged.

	\end{enumerate}

\subsubsection*{Acknowledgements}

We thank Vaughan Jones for stimulating discussions and for bringing the references \cite{Gurevich:1986,Gurevich:1991} to our attention. We also thank
James Borger for an interesting exchange about \(\lambda\)-operations on
Thoma parameters. Furthermore, the hospitality of the Isaac Newton Institute for Mathematical Sciences in Cambridge during the program ``Operator algebras: subfactors and their applications'', supported by EPSRC Grant Number EP/K032208/1, is gratefully acknowledged.

\bigskip

\noindent\footnotesize 
Gandalf Lechner: {\tt LechnerG@cardiff.ac.uk}\\[2mm]
Ulrich Pennig: {\tt PennigU@cardiff.ac.uk}\\[2mm]
Simon Wood: {\tt WoodSi@cardiff.ac.uk}

\end{document}